\documentclass[11pt]{article}
\usepackage[utf8]{inputenc}
\usepackage[T1]{fontenc}
\usepackage{lmodern}
\usepackage{amsmath}
\usepackage{amsfonts}
\usepackage{amsthm}
\usepackage{amssymb}
\usepackage{dsfont}
\usepackage{blindtext}
\usepackage{titlesec}
\usepackage{mathtools}
\usepackage{epigraph}
\usepackage{enumitem}
\usepackage{geometry}
\usepackage[colorlinks=true, linkcolor=orange, citecolor=blue]{hyperref}
\usepackage{mathtools}
\usepackage{hyperref}
\usepackage{tabularx}
\usepackage{comment}
\usepackage{bigints}
\usepackage{mathrsfs}

\usepackage{tikz}
\usetikzlibrary{shapes,positioning,intersections,quotes,patterns}
\usepackage{graphics}
\usepackage{graphicx}
\usetikzlibrary[topaths]
\usepackage{pgfplots}

\usepackage{caption}
\usepackage{subcaption}

\usepackage{mathtools}
\mathtoolsset{showonlyrefs}

\usepackage{caption}

\newcommand{\ct}{\mathfrak{C}}
\newcommand{\dist}{\textnormal{dist}}

\newcommand{\tchi}{\tilde{\chi}}

\newcommand{\M}{\mathcal{M}}

\newcommand{\ve}{\varepsilon}
\newcommand{\pp}{e^{-\frac{\ve D^2_t}{2 \tau}}}

\newcommand{\dle}{\Box_{\ell, \ve}}
\newcommand{\rt}{\tilde{R}}
\newcommand{\rd}{\tilde{r}_1}
\newcommand{\g}{\gamma}
\renewcommand{\a}{\mathfrak{a}}
\renewcommand{\b}{\mathfrak{b}}
\newcommand{\chil}{\chi_\lambda}
\newcommand{\vs}{\varsigma}
\renewcommand{\Im}{\textnormal{Im}}
\renewcommand{\Re}{\textnormal{Re}}
\newcommand{\kp}{\kappa^\prime}
\newcommand{\ttheta}{\tilde{\theta}}
\newcommand{\D}{\mathcal{D}}
\newcommand{\C}{\mathcal{C}}
\newcommand{\rtt}{\tilde{r}}

\usepackage{mathrsfs} 
\usepackage{textcomp}
\usepackage[T1]{fontenc}
\usepackage{newunicodechar}
\numberwithin{equation}{section}

\theoremstyle{plain}
\newtheorem{thm}{Theorem}[section]
\newtheorem{lem}[thm]{Lemma}
\newtheorem{prop}[thm]{Proposition}
\newtheorem{cor}[thm]{Corollary}

\theoremstyle{definition}
\newtheorem{remark}[thm]{Remark}
\newtheorem{definition}{Definition}

\title{On the blowup of quantitative unique continuation estimates for waves and applications to stability estimates}
\author{Spyridon Filippas\footnote{Department of Mathematics and Statistics, University of Helsinki, Helsinki, Finland, email: spyridon.filippas@helsinki.fi}, Lauri Oksanen\footnote{Department of Mathematics and Statistics, University of Helsinki, Helsinki, Finland, email: lauri.oksanen@helsinki.fi}}
\date{\today}

\makeatletter
\def\keywords{
	\vspace{1ex}
	\noindent
	\if@twocolumn
	\small{\bf  Keywords}\/---$\!$    \else
	\begin{center}\small\ {\bf Keywords}\end{center}\quotation\small
	\fi}
\def\endkeywords{\vspace{0.6em}\par\if@twocolumn\else\endquotation\fi
	\normalsize\rm}
\makeatother

\def\p{\partial}
\def\R{\mathbb R}
\DeclareMathOperator{\supp}{supp}
\newcommand{\norm}[3]{\left\Vert #1 \right\Vert_{#2}^{#3}}

\renewcommand{\div}{\textnormal{div}}

\pgfplotsset{compat=1.18}

\begin{document}
	
	\maketitle
	
	\begin{abstract}
	   In this paper we are interested in the blowup of a geometric constant $\ct(\delta)$ appearing in the optimal quantitative unique continuation property for wave operators. In a particular geometric context we prove an upper bound for $\ct(\delta)$ as $\delta$ goes to $0$. Here $\delta>0$ denotes the distance to the maximal unique continuation domain. As applications we obtain stability estimates for the unique continuation property up to the maximal domain. Using our abstract framework~\cite{FO25abstract} we also derive a stability estimate for a hyperbolic inverse problem. The proof is based on a global explicit Carleman estimate combined with the propagation techniques of~\cite{Laurent_2018}.
	\end{abstract}
	

	\tableofcontents
	
	\section{Introduction}

\subsection{Background}

In this paper, motivated by applications in inverse problems and control theory, we are interested in studying the behavior of a particular constant related to the property of unique continuation for the wave operator. Generally speaking, given a differential operator $P$ the question of unique continuation consists in asking the following: suppose that $u$ satisfies $Pu=0$ on some domain $\Omega$. Does the knowledge of $u$ on some subset $U \subset \Omega$ determine $u$ \textit{everywhere}? 
If this property holds, then the next natural question is if we can \textit{quantify} it. This is expressed via a stability estimate of the form
\begin{equation}
\label{quantitative phi}
\norm{u}{\Omega}{} \lesssim \theta\left(\norm{u}{U}{},\norm{Pu}{\Omega}{},\norm{u}{\Omega}{}\right),    
\end{equation}
with $\theta$ satisfying
$$
\theta(a,b,c) \stackrel{a,b \rightarrow0}{\longrightarrow} 0, \quad \textnormal{with }c\textnormal{ bounded.}
$$
Such an estimate is also called a conditional estimate, since one needs to make the assumption that $\norm{u}{\Omega}{}$ remains bounded and it reflects the ill posed nature of the problem under consideration.

In the case where $P$ is the wave operator the question is of particular interest when the observation set $U$ is a cylinder of the form $(-T,T) \times \omega$ and  $\Omega=(-T,T) \times \M$. In this case one wants to know if the observation of the wave from the set $\omega$ during a time $2T$ is sufficient to determine it in the whole domain $\M$. In one assumes that \textit{all} the coefficients of $P$ are analytic then the Holmgren-John theorem~\cite{John:49} gives a positive answer as soon as $T > \mathcal{L}(\M, \omega)$ where we denote by $$\mathcal{L}(\M, \omega)= \sup_{x \in M}\textnormal{dist}(x,\omega),$$ the largest distance of the subset $\omega$ to a point of $M$. Due to finite speed of propagation this time of observation is in fact optimal, see e.g~\cite{russell1971boundarya,russell1971boundaryb}.

However, once analyticity is no longer assumed the problem becomes significantly more complicated. A first seminal result for the wave equation was achieved in~\cite{Robbiano:91} where the property of unique continuation from $(-T,T) \times \omega$ is proved as soon as $\omega \neq \emptyset$ and the time $T>0$ is sufficiently large. The observation time $T$ was subsequently improved in~\cite{Hormander:92}. It is in the breakthrough by Tataru~\cite{Tataru:95} that unique continuation in \textit{optimal time} was finally proved. The regularity assumption for the coefficients of the operator is a partial analyticity condition. In particular, this gives an optimal unique continuation result for wave operators whose coefficients are independent of time. This and some follow up works eventually led to a very general unique continuation result which applies to operators with partially~\textit{analytic} coefficients: the Tataru-Robbiano-Zuily-Hörmander~\cite{Tataru:95, tataru1999unique, RZ:98,Hor:97} Theorem. In the case of the wave operator the condition is analyticity of all the coefficients with respect to time. Moreover, Alinhac and Alinhac-Bahouendi~\cite{AB:79,Alinhac:83,AB:95} constructed counterexamples if the coefficients are considered to be only smooth with respect to time. Those were later refined by H\"ormander~\cite{Hormander:00} who proved that in fact even a Gevrey $s$ regularity condition for any $s>1$ is not enough to guarantee uniqueness. This shows that the analyticity condition is essentially optimal (see~\cite{LL:23notes} for more details in the case of the wave operator). We also refer to the recent article~\cite{FLL:24} where the analyticity condition was relaxed to Gevrey $2$ for some time dependent Schrödinger operators. 

Concerning the \textit{quantitative} counterpart of the above, a stability estimate was proved by Bossi-Kurylev-Lassas in~\cite{BKL:16, bosi2018stability, bosi2017reconstruction} and Burago-Ivanov-Lassas-Lu~\cite{burago2024quantitativestabilitygelfandsinverse} with 
\begin{align}
\label{est_log_alpha}
    \theta(a,b,c)=c \log \left(1+\frac{c}{a+b}\right)^{-\alpha}, \alpha \in (0,1).
\end{align} 
At the same time Laurent-Léautaud quantified in~\cite{Laurent_2018} the general theorem of Tataru-Robbiano-Zuily-Hörmander theorem. Their estimate is of the form~\eqref{est_log_alpha} with $\alpha=1$. More precisely, in the case of the wave operator $\partial_t^2-\Delta+q$ with $q$ a time independent potential, they prove the following. For any open subset $\omega \subset \M$ of a compact Riemannian manifold $(\M, g)$, for $u$ satisfying $(\p^2_t-\Delta_g+q)u=f$ and any $T > \mathcal{L}(\M, \omega)$ one has, for some $\ve>0$,
\begin{align}
\label{est_ll_intro}
  \norm{u}{L^2((-\ve,\ve) \times \M)}{}  \leq \ct \frac{\norm{u}{H^1((-T,T)\times \M)}{}}{ \log \left(1+\frac{\norm{u}{H^1((-T,T)\times \M)}{}}{\norm{u}{L^2((-T,T) \times H^1(\omega))}{} +\norm{f}{L^2((-T,T)\times \M)}{}}\right)}.
\end{align}
The constant $\ct$ in~\eqref{est_ll_intro} depends on $T, \M, \omega, $ and $\norm{q}{L^\infty}{}$. The log stability above is in fact \textit{optimal} as shown by Lebeau in~\cite{Leb:Analytic} (see as well Theorem 6.10 in~\cite{KRL:11} for an instability result). In particular,~\eqref{est_log_alpha} is suboptimal due to $\alpha<1$. It should be emphasized however that similarly to the spirit of the present paper and contrary to~\cite{Laurent_2018}, in~\cite{BKL:16, bosi2017reconstruction, burago2024quantitativestabilitygelfandsinverse} the authors give bounds for the constant $c$ in~\eqref{est_log_alpha}.

\subsection{Main results and comments}
\label{sec_main_results}

In this article, we are interested in understanding the behavior of the optimal quantitative unique continuation constant $\ct$ in~\eqref{est_ll_intro} in a regime where $T-\mathcal{L}(\M, \omega)>0$ goes to $0$. For fixed $\M, \omega$ this amounts to studying $\ct$ as the observation time approaches the critical one. However, one can also fix $\M, \omega$ and define $T=\mathcal{L}(\M, \omega)$. Then, estimate~\eqref{est_ll_intro} still holds if one replaces the norm $\norm{u}{L^2((-\ve,\ve) \times \M)}{}$ by some norm in a smaller domain $\M_\delta \Subset \M$ so that $T>\mathcal{L}(\M_\delta, \omega)$. In this case, in the limit $\delta \to  0$ one looks at the behavior of $\ct$ as $\M_\delta$ approaches the maximal domain where the solution has to vanish. This corresponds to reaching exactly the threshold where unique continuation no longer holds and consequently $\ct$ is expected to blow up.

\bigskip

Let us now state our main results. Let $r_0>1$ and consider $R \in [\frac{1}{r_0}, r_0]$. Define the cylinder $\C \subset \R^{1+n}=\R_t \times \R^n_x$ by 
\begin{align}
\label{def_cylinder}
\C=\{(t,x) | \: |x| < R, |t| < \frac{R}{2}  \},
\end{align}
and the diamond $\D$
\begin{align}
    \label{def_diamond}
    \D=\{(t,x) | \: |t|<\frac{3R}{2} -|x|, |t|<\frac{R}{2}  \}.
\end{align}

By qualitative unique continuation one has that a solution of the wave equation that vanishes in the cylinder $\C$ has to vanish in the whole diamond $\D$, see Figure~\ref{cone and delta}. Moreover, the diamond $\D$ is the maximal domain in which such a solution has to vanish. Our main result gives an~\textit{optimal} stability estimate of this with an explicit dependence of the constant as one approaches $\D$.

Let us now define smaller domains that approach $\D$. Consider the function $\phi$ defined by $2\phi=-t^2+(|x|-\frac{3R}{2} )^2$ and write $\M_\delta=\{ \phi > \delta^2 \}$. The function $\phi$ gives a foliation of the cone $ |t|<\frac{3}{2} R -|x| $, as shown in Figure~\ref{fig_foliation}. Notice in particular that $\D=\{ \phi>0\} \cap \{t \in [-\frac{R}{2}, \frac{R}{2}]\}$. We denote by $\Box=\p^2_t-\Delta$ the wave operator.

\begin{thm}[Optimal stability estimate arbitrarily close to the maximal domain]
\label{thm_log_close_tocone}
  Consider $q \in L^\infty(\D)$ with $\norm{q}{L^\infty(\D)}{} \leq M$. There exists $N >0 $ depending on $r_0,n,M $ only, such that the following holds. Let $\delta>0$ small and define $\D_\delta= \D \cap \M_\delta$. Then for any $u \in H^1(\D)$ with $(\Box+q) u = f \in L^2$ on $\D$ one has
\end{thm}
\begin{align}
\norm{ u}{L^2(\D_{\delta})}{} \leq \ct(\delta) \frac{\norm{u}{H^1(\D)}{}}{\log\left(1+\frac{\norm{u}{H^1(\D)}{}}{\norm{u}{H^1(\C)}{}+\norm{f}{L^2(\D)}{}}\right)},
\end{align}
where 
\begin{align}
\label{the_coste}
\ct(\delta)=\left(\frac{1}{\delta }\right)^{\frac{N}{\delta^4}}.
\end{align}

Theorem~\ref{thm_log_close_tocone} is optimal in two ways:

\begin{itemize}
    \item The domain $\D$ is the maximal one, for otherwise finite speed of propagation for waves would be violated. 

    \item The log stability estimate is the optimal one.
\end{itemize}

The constant $\ct(\delta)$ can be seen as an approximate observability constant for the wave operator. By duality, unique continuation is equivalent to approximate observability, and as proved in~\cite{Robbiano:95}, a quantitative unique continuation result gives an estimate on the cost of the control. See as well the discussion in~\cite[Section 2]{FO25abstract}. Studying the behavior of geometric constants in control theory has been an active field of research, especially in the case of the heat equation. It is known since the seminal papers of Lebeau-Robbiano~\cite{LR:95} and Fursikov-Imanuvilov~\cite{FI:96} that for any time $T>0$ and any open $\omega$ the heat equation is exactly observable. The blow up of the associated observability constant with respect to $(\omega, T)$ has been the object of several studies. We refer to the introduction in~\cite{LL:18} for more details concerning geometric constants in control theory and their links to each other. 

Concerning the wave operator, in~\cite{LL:16} Laurent-Léautaud give bounds for the blow up of the~\textit{exact} observability constant for waves as the time tends to the minimal exact observability time. It is known since the pioneering work of Bardos–Lebeau–Rauch~\cite{BLR:92} that the necessary and sufficient condition for exact observability for waves is the Geometric Control Condition. In~\cite{LL:16} the authors use in a crucial way the optimal quantitative estimate~\eqref{est_ll_intro}. However, the geometric control time is larger than the unique continuation time. As a consequence, as far as the blow up of the \textit{exact} observability constant is concerned, one does not need to study the blow of the~\textit{approximate} observability constant $\ct(\delta)$. Theorem~\ref{thm_log_close_tocone} is, to the best of our knowledge, the first result giving a bound for the optimal approximate observability constant for waves.

\bigskip

Additionally, Theorem~\ref{thm_log_close_tocone} can be combined with an optimization argument to derive a stability estimate up to the~\textit{whole diamond}.

\begin{thm} [Stability estimate up to the optimal domain]
   \label{thm_up_to_cone}
  Consider $q \in L^\infty(\D)$ with $\norm{q}{L^\infty(\D)}{} \leq M$. There exists $C>0$ depending on $r_0, M,n$ only, such that for $u \in H^1(\D)$ satisfying $(\Box u+q) =f \in L^2$ on $\D$ one has the following stability estimate
 \begin{align}
     \norm{u}{L^2(\D)}{} \leq \frac{C  \norm{u}{H^1(\D)}{}}{\left(\log\left(\log\left(1+\frac{\norm{u}{H^1(\D)}{}}{\norm{u}{H^1(\C)}{}+\norm{f}{L^2(\D)}{}}\right)\right) +1\right)^{\frac{4}{15}}}.
 \end{align}
\end{thm}

\begin{remark}
    The result of Theorem~\ref{thm_up_to_cone} is a $\log \log$ stability estimate and should be interpreted as follows. Suppose that $u$ satisfies $(\Box +q)u =0$, $\norm{u}{H^1}{}=1$ and $\norm{u}{H^1(\C)}{}= \ve$. Then there is $C>0$ and $\ve_0>0$ depending only on $r_0,M,n$ such that for $0< \ve \leq \ve_0$ one has
    \begin{align}
      \norm{u}{L^2(\D)}{} \leq  \frac{C}{\left(\log |\log \ve| \right)^{\frac{4}{15}}}.
    \end{align}
\end{remark}

It is interesting to compare the result of Theorem~\ref{thm_up_to_cone} with that of conditional stability estimates for the elliptic unique continuation problem. It is known since the work of Hadamard~\cite{hadamard1902problemes} that this problem is ill posed in the sense that its solution may not depend continuously on the data. However, a stability estimate can be proved if one adds some a priori assumption on the solution $u$. This is a conditional stability of the form~\eqref{quantitative phi}. In the case where $\omega\subset \Omega$ is the observation set and $V \Subset \Omega$ is an open set that~\textit{stays away} from the boundary of $\Omega$ one obtains a Hölder type estimate, that is, an estimate of the form $$\norm{u}{H^1(V)}{} \lesssim \norm{u}{H^1(\Omega)}{1-\delta}(\norm{\Delta u}{L^2(\Omega)}{}+\norm{u}{L^2(\omega)}{})^\delta.$$ This can be seen as the analogue of~\eqref{est_ll_intro} in the elliptic case. However, if one wants to obtain a stability estimate~\textit{up to the whole} domain $\Omega$ then the estimate deteriorates to logarithmic, that is $$\norm{u}{H^1(\Omega)}{}\lesssim \frac{\norm{u}{H^2(\Omega)}{}}{\log\left(1+\frac{\norm{u}{H^2(\Omega)}{}}{\norm{\Delta u}{L^2(\Omega)}{}+\norm{u}{L^2(\omega)}{}}\right)}.$$

The above result originates in the work of Phung~\cite{phung2003remarques} with subsequent improvements in the works of Bourgeois~\cite{bourgeois2010stability} and Bourgeois-Dardé~\cite{bourgeois-darde2010stability}. Moreover, it is proven by Bourgeois~\cite{bourgeois2017quantification} in the two dimensional setting that the logarithmic estimate is essentially optimal. For more details on the stability in the elliptic case see~\cite[Chapter 5]{rousseau2022elliptic} as well as~\cite{alessandrini2009stability}.

The loglog stability obtained in Theorem~\ref{thm_up_to_cone} may seem weak. Nevertheless, it is analogous to the elliptic case in the sense that stability deteriorates as one approaches the maximal domain. In the elliptic case, from Hölder stability in the interior of the domain one derives log stability in the whole domain, and both of the stabilities are optimal. In the wave case, we prove that log stability in the interior of the maximal domain yields loglog stability in the whole maximal domain. The former estimate is optimal, but it remains an open question whether this is the case for the latter one as well.

\bigskip

Finally, a particular motivation to study $\ct(\delta)$ comes from inverse problems. Unique continuation for waves is one of the main ingredients of the Boundary Control method. The Boundary Control method, originating in~\cite{Belishev:87}, can be used to solve the problem of recovering a metric or a potential from the Dirichlet-to-Neumann map (or variants of it) associated to the wave operator. In some geometric situations the Boundary Control method is the only known way to prove uniqueness. Moreover, quantitative unique continuation results can be used to obtain stability estimates in situations where Boundary Control method applies~\cite{bosi2017reconstruction, burago2024quantitativestabilitygelfandsinverse, FO25abstract}. However, a quantitative unique continuation result such as~\eqref{est_ll_intro} is not enough. One needs additionally to know how the constant degenerates when one approaches the maximal domain of unique continuation. This comes, roughly speaking, from the fact that Boundary Control method relies crucially on unique continuation \textit{exactly} up to the maximal domain that is allowed by finite speed propagation.

In the companion paper~\cite{FO25abstract} we establish an abstract link between optimal unique continuation and the stability of an inverse problem for the wave equation. Theorem~\ref{thm_up_to_cone} allows to prove a concrete stability estimate for the recovery of a potential on some compact region from its source to solution map when the observation takes place on a half-space. More precisely, we denote by $H$ the half-space $\{x_1>0\}$ in $ \R^n_x$. Let $q_{j} \in C^\infty (\R^n; \R), j\in\{1,2\}$ and let $K$ be a compact set of $\{x_1<0\}$. For a source $F \in C^\infty_0 ((0,T)\times H)$ we define the source to solution map $\Lambda_{q_j}$ associated to the potential $q_j$ by
$$
\Lambda_{q_j}(F)=u_{|(0,T) \times H},
$$
where $u$ is the solution of
\begin{equation}
\begin{cases}
    (\Box+q)u=F & \textnormal{in} \:(0,T) \times \R^n,\\
    u_{|t=0}=\p_tu_{|t=0}=0  & \textnormal{in} \: \R^n.
    \end{cases}
\end{equation}
In~\cite{FO25abstract} we use Theorem~\ref{thm_up_to_cone} to prove the following stability estimate.

\begin{thm}[Theorem 1.2 in \cite{FO25abstract}]
\label{thm_inverse_half}
Assume that the potentials $q_j$ satisfy $\norm{q_j}{C^m(\R^n)}{}\leq M, j \in \{1,2\}$ for some $M>0$ and $m >n$. Then there is $T_0$ depending only on $K$ such that for $T>T_0$ there exist $C>0$ and $\alpha \in (0,1)$ depending on $T,K,M,n$ with 
\begin{align*}
    \norm{q_2-q_1}{L^2(K)}{}\leq \frac{C}{\left(\log |\log \norm{\Lambda_1-\Lambda_2}{L^2\to L^2}{} |\right)^\alpha}.
\end{align*}
\end{thm}

 \bigskip

   \begin{figure}
	\centering
	
	\begin{tikzpicture}
\draw (0,-2) -- (0,2) node[above] {$t$} ;
\draw (-5,0) -- (5,0) node[right] {$x$};

\draw (-4,0) -- (-2,1) ;
\draw (-2,1) -- (2,1) ;
\draw (2,1) -- (4,0) ;

\draw (4,0) -- (2,-1) ;
\draw (2,-1) -- (-2,-1) ;
\draw (-2,-1) -- (-4,0) ;


\fill[fill=orange!30,semitransparent] (-4+0,0) -- (-2,1-0)-- (-2,-1+0)--(-4+0,0) ;

\fill[fill=orange!30,semitransparent] (4+0,0) -- (2,1-0)-- (2,-1+0)-- (4+0,0) ;

\filldraw[blue!30, semitransparent]  (-2,-1) rectangle (2,1) ;
\draw [blue, thick] (-2,0) -- (2,0);

\node at (1.2,0.2) [blue] {$\omega$};

\node at (-3,1.5) [orange!60] {$\D$};

\filldraw[black] (0,1) circle (1pt) node[anchor=south east]{$s$};
\filldraw[black] (0,-1) circle (1pt) node[anchor=south east]{$-s$};

\end{tikzpicture}
\caption{A solution of the wave equation that vanishes on $(-s,s)\times \omega$ has to vanish in the diamond $\D$. Theorem~\ref{thm_up_to_cone} gives a stability estimate~\textit{in the whole diamond} in the case where $s=R/2$ and $\omega$ is a ball of radius $R$.}
\label{cone and delta}
\end{figure}
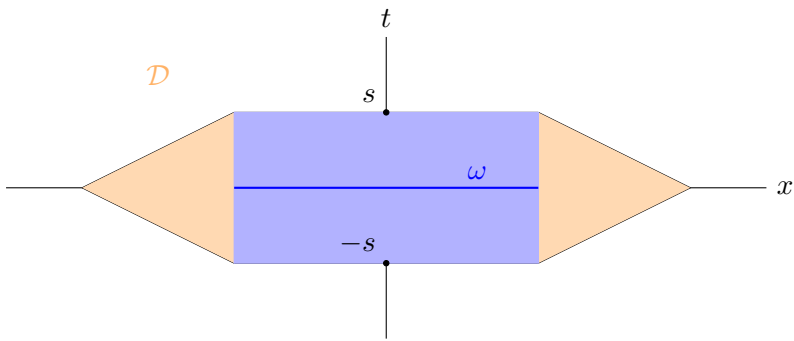

    \subsection{Idea of the proof and plan of the paper}

 Since the pioneering work of Carleman~\cite{carleman1939probleme}, Carleman estimates are one of the main tools in proving qualitative as well as quantitative uniqueness. These are weighted estimates of the form
 \begin{align}
     \norm{e^{\tau \phi }u}{L^2}{}\lesssim   \norm{e^{\tau \phi }Pu}{L^2}{},
 \end{align}
applied to compactly supported functions. The weight $e^{\tau \phi}$ allows to propagate uniqueness along level sets of $\phi$ in the direction where $\phi$ decreases by letting  $\tau \to \infty$. The key additional idea introduced by Tataru in~\cite{Tataru:95} (following the introduction of an FBI transform in~\cite{Robbiano:91}) is to replace $e^{\tau \phi}$ by $e^{\tau \phi} \pp$ where the weight $\pp$ is a Fourier multiplier which essentially localizes in a low frequency regime. The introduction of this non local weight allows one to obtain unique continuation in optimal time for the wave equation. At the same time, it makes the proof and application of Carleman estimates much more intricate. 

For the problem under consideration here, we need a Carleman estimate containing this Gaussian weight but we additionally need~\textit{an explicit dependence of the constants} with respect to the geometry of the problem. More precisely, the crucial parameter here is the distance $\delta$ to the optimal domain. Since the proofs of the Carleman estimates in~\cite{Tataru:95,tataru1999unique, RZ:98, Hor:97} are based on microlocal arguments, determining explicit constants in this general setting would be extremely tedious. Moreover, all the proofs are based on local estimates close to a point, that need to be iterated in order to cover the whole domain. This makes the tracking of the constants even more difficult. 

In order to address those problems, we asked the following question: is it possible, at least in some specific geometric situation, to prove a~\textit{global} Carleman estimate adapted to the optimal unique continuation for the wave operator? That means that this estimate should contain the non local weight $\pp$. It turns out, that we are able to do that in the geometry of Figure~\ref{cone and delta}. Our global Carleman estimate is based on explicit calculations carried out on the physical space (and not on the space of symbols).

Heuristically, we expect that the more one needs to iterate an estimate, the more the constants degenerate. As a consequence, having a global Carleman estimate is a good starting point for obtaining good constants. Ideally, one can even hope that no iteration is needed in this case and that we only need to apply the Carleman estimate once. This would be indeed the case if we had a classical Carleman estimate. As we prove in this paper, the presence of the non local weight $\pp$ makes this impossible and essentially forces us to iterate our Carleman estimate. This eventually gives the exponential blow up of $\ct(\delta)$.

Once we have obtained the Carleman estimate we use the propagation techniques of~\cite{Laurent_2018} to quantify the uniqueness property. The new feature is that we need to keep track of the dependency with respect to $\delta$. While the propagation is quite technical, the overall proof is transparent enough to allow us to understand the bottlenecks that lead to the precise blow up of $\ct(\delta)$. See as well Remarks~\ref{rem_espilon_depends_ongamma}, \ref{remark_dependancy on g} and \ref{rem_numberofiterations}.

\bigskip

The plan of the paper is as follows. In Section~\ref{sec_carleman} we prove the global Carleman estimate, that is Theorem~\ref{thm_explicit_carleman}. In Section~\ref{sec_local} we combine the Carleman estimate with the techniques of~\cite{Laurent_2018} to obtain an explicit quantitative estimate that can be iterated. In Section~\ref{sec_iteration} we iterate the quantitative estimate following the level sets of the foliating function $\phi$. Finally, in Section~\ref{sec_applications} we prove all the main results that we stated in the introduction of the paper. Appendix~\ref{sec_append} collects several lemmata that are used throughout the paper.   

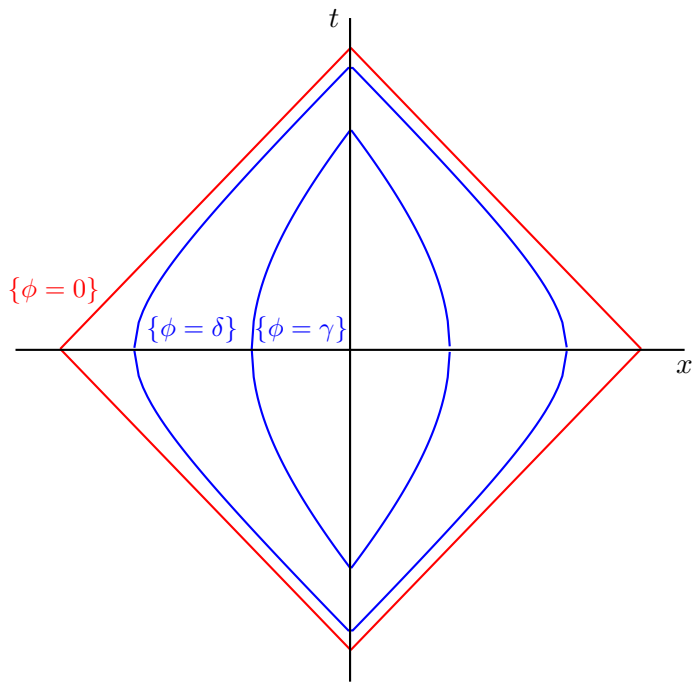
\begin{figure}
    \centering

    \begin{tikzpicture}
\begin{axis}[
    axis lines=none, 
    domain=-0.5:0.5, 
    samples=100,
    width=12cm, 
    height=12cm, 
    axis line style={-}, 
    ymin=-0.8, ymax=0.8, 
    xtick=\empty, 
    ytick=\empty, 
    xlabel={}, 
    ylabel={}, 
    axis equal, 
    restrict y to domain=-1.5:1.5, 
    enlarge x limits={abs=0.5}, 
    xmin=-1.5, xmax=1.5, 
]

    \addplot[blue, thick] {sqrt((3/2 - abs(x))^2 - 1)}; 
    \addplot[blue, thick] {-sqrt((3/2 - abs(x))^2 - 1)}; 
\end{axis}

\begin{axis}[
    axis lines=none, 
    domain=-1.0917:1.0917, 
    samples=100,
    width=12cm, 
    height=12cm, 
    axis line style={-}, 
    ymin=-0.8, ymax=0.8, 
    xtick=\empty, 
    ytick=\empty, 
    xlabel={}, 
    ylabel={}, 
    axis equal, 
    restrict y to domain=-1.5:1.5, 
    enlarge x limits={abs=0.5}, 
    xmin=-1.5, xmax=1.5, 
]

    \addplot[blue, thick] {sqrt((3/2 - abs(x))^2 - 1/6)}; 
    \addplot[blue, thick] {-sqrt((3/2 - abs(x))^2 - 1/6) }; 

     \node at (axis cs:-0.25, 0.1) { \textcolor{blue}{\small $\{\phi = \gamma\}$}}; 
      \node at (axis cs:-0.8, 0.1) { \textcolor{blue}{\small $\{\phi =\delta \}$}};

       \node at (axis cs:-1.5, 0.3) { \textcolor{red}{\small $\{\phi = 0\}$}};
    \end{axis}

\begin{axis}[
    axis lines=none,
    xmin=-1.5, xmax=1.5,
    ymin=-1.5, ymax=1.5,
    width=12cm, height=12cm,
    axis equal,
]
    \addplot[red, thick, domain=0:1, samples=2] coordinates {(1.1,0) (0,1.15)};
    \addplot[red, thick, domain=0:1, samples=2] coordinates {(0,1.15) (-1.1,0)};
    \addplot[red, thick, domain=0:1, samples=2] coordinates {(-1.1,0) (0,-1.15)};
    \addplot[red, thick, domain=0:1, samples=2] coordinates {(0,-1.15) (1.1,0)};
\end{axis}

\def\xLength{4.4}  
\def\yLength{4.4}  

\def\xOffset{5.2}  
\def\yOffset{5.2}  

\draw[thick] (-\xLength + \xOffset, \yOffset) -- (\xLength + \xOffset, \yOffset)  node[anchor=north] {$x$};

\draw[thick] (\xOffset, -\yLength + \yOffset) -- (\xOffset, \yLength + \yOffset)  node[anchor=east] {$t$} ;

\end{tikzpicture}

    \caption{The level sets of the foliating function $\phi$, which approach the boundary of the cone $\{\phi>0\}$ as $\delta$ goes to $0$.  }
    \label{fig_foliation}
\end{figure}

\section{The global Carleman estimate}
\label{sec_carleman}

\subsection{A ``Carleman equality''}

	Let $(M,g)$ be a semi-Riemannian manifold. 
	We write $(X,Y) = g(X,Y)$ for the scalar product of two vector fields $X$ and $Y$ on $M$. The associated quadratic form is denoted by $G(X) = g(X,X)$. We will need the following elementary formula for 
	functions $u$, $v$ and a vector field $X$ on $M$,
	\begin{align}\label{hessian}
		X (\nabla u, \nabla v) = D^2 u(X, \nabla v) + D^2 v(X, \nabla u), 
	\end{align}
	where $D^2$ stands for the Hessian.

	\begin{lem}
		\label{lem_carleman_eq}
		Let $\ell, \sigma, v \in C^2(M)$. Then
		\begin{align}
			|e^\ell \Delta (e^{-\ell} v)|^2/2 
			&= Q_+(\nabla v) + Q_-(\nabla \ell) v^2
			+ \div B + R
			\\&\qquad
			+ |(\Delta + q + \sigma) v|^2/2 
			+ |(L + \sigma) v|^2/2,
		\end{align}
		where, writing $a = \sigma - \Delta \ell$, the quadratic forms $Q_\pm$ are defined by 
		\begin{align}
			Q_\pm(X) = 2 D^2 \ell(X,X) \pm ag(X,X),
		\end{align}
		and the following notation is used $q = G(\nabla \ell) - \Delta \ell$, $L = 2 \nabla \ell$,
		\begin{align}
			B 
			&= 
			-((L + \sigma) v)\nabla v 
			+ (G(\nabla v) - (q + \sigma) v^2)\nabla \ell 
			+ v^2 \nabla \sigma/2,
			\\
			R 
			&=
			(\div(a \nabla \ell) - a \sigma - \Delta \sigma/2)v^2.
		\end{align}
	\end{lem}
	\begin{proof}
		We have
		\begin{align}\label{expansion}
			e^\ell \Delta (e^{-\ell} v)
			= 
			(\Delta -L + q)v.
		\end{align}
		Indeed, using 
		\begin{align}
			e^\ell \nabla (e^{-\ell} v) = \nabla v - v \nabla \ell, 
			\quad 
			e^\ell \div (e^{-\ell} X) = \div X - (\nabla \ell, X),
		\end{align}
		we obtain
		\begin{align}
			e^\ell \Delta (e^{-\ell} v)
			&= 
			e^\ell \div (e^{-\ell} (\nabla v - v \nabla \ell))
			\\&= 
			\Delta v 
			- \div(v \nabla \ell) 
			- (\nabla \ell, \nabla v)
			+ (\nabla \ell, \nabla \ell) v
			\\&=
			\label{terms_expansion}
			\Delta v 
			- 2(\nabla \ell, \nabla v)
			- v \Delta \ell
			+ (\nabla \ell, \nabla \ell) v.
		\end{align}
		Thus \eqref{expansion} holds, and
		\begin{align}\label{squared}
			|e^\ell \Delta (e^{-\ell} v)|^2/2
			&= 
			|(\Delta + q + \sigma) v - (L + \sigma) v|^2/2
			\\\notag&= 
			|(\Delta + q + \sigma) v|^2/2 
			+ |(L + \sigma) v|^2/2,
			\\\notag&\qquad 
			- \Delta v(L + \sigma) v
			- (q + \sigma) v (L + \sigma) v.
		\end{align}
		Let us study the last two terms in \eqref{squared}, that we call the cross terms. 
		
		For the first cross term we have
		\begin{align}
			- \Delta v(L + \sigma) v
			= 
			- \div(((L + \sigma) v)\nabla v) + \nabla v(L v) + \nabla v(\sigma v).
		\end{align}
		Recalling that $Lv = 2 (\nabla \ell, \nabla v)$ and using \eqref{hessian} with $X = \nabla v$, 
		we have
		\begin{align}
			\nabla v(Lv) = 2 D^2 \ell(\nabla v, \nabla v) + 2 D^2 v(\nabla v, \nabla \ell).
		\end{align}
		On the other hand, using \eqref{hessian} with $X = \nabla \ell$, we have
		\begin{align}\label{Hess_expansion}
			2 D^2 v(\nabla \ell, \nabla v) = \nabla \ell (\nabla v, \nabla v)
			= \div(G(\nabla v)\nabla \ell) - G(\nabla v) \Delta \ell.
		\end{align}
		Finally,
		\begin{align}
			\nabla v(\sigma v) 
			= 
			\frac12 \nabla \sigma(v^2) + \sigma G(\nabla v)
			=
			\frac12 \div(v^2 \nabla \sigma) - \frac12 v^2 \Delta \sigma + \sigma G(\nabla v).
		\end{align}
		We collect the above and obtain
		\begin{align}
			- \Delta v(L + \sigma) v
			= Q_+(\nabla v) + \div B_1 + R_1,
		\end{align}
		where
		\begin{align}
			B_1 
			= 
			-((L + \sigma) v)\nabla v 
			+ G(\nabla v) \nabla \ell 
			+ \frac12 v^2 \nabla \sigma,
			\quad
			R_1 
			= 
			-\frac12 (\Delta \sigma) v^2.
		\end{align}
		
		We write $\tilde q = q + \sigma = G(\nabla \ell) + a$, and have for the second cross term
		\begin{align}
			- \tilde q v (L + \sigma) v 
			&= 
			-\frac12 \tilde q L(v^2)
			- \sigma \tilde q v^2
			\\&=
			-\frac12 \div(v^2 \tilde q L) 
			+ \frac12 \div(\tilde q L) v^2
			- \sigma \tilde q v^2
			\\&=
			(\div (G(\nabla \ell) \nabla \ell) 
			- \sigma G(\nabla \ell)) v^2 
			+ \div B_2 + R_2,
		\end{align}
		where
		\begin{align}
			B_2 = - v^2 \tilde q \nabla \ell, 
			\quad 
			R_2 = (\div(a \nabla \ell) - a \sigma)v^2.
		\end{align}
		Using \eqref{Hess_expansion} with $v = \ell$ gives
		\begin{align}
			\div (G(\nabla \ell) \nabla \ell) - \sigma G(\nabla \ell)
			=
			2 D^2 \ell(\nabla \ell, \nabla \ell) -a G(\nabla \ell) = Q_-(\nabla \ell).
		\end{align}
		The claim follows by observing that $B = B_1 + B_2$ and $R = R_1 + R_2$.
	\end{proof}

	\subsection{The subelliptic estimate}

	We now focus on the explicit situation where $M = \R^{1+n}$ and $g = -dt^2 + dx^2$ is the Minkowski metric on $M$. We write $\Box = \Delta$. We use the polar coordinates $(r,\theta)$ on the spatial factor $\R^n$. For some $r_0>1$ fixed we consider $\rtt \in [\frac{1}{r_0}, r_0]$. We consider as well $\ell = \tau \phi$ where 
	\begin{align*}
		2 \phi = -t^2 + 2 \rho = -t^2 + (r - \rd)^2,
	\end{align*}
	will be the foliating function and $\rd > 0$ to be determined. We start by some calculations that are needed in order to compute the important quantities $Q_\pm$ of Lemma~\ref{lem_carleman_eq}.

    \bigskip

    \textbf{Notation:} \textit{In the sequel we shall use the notational convention of writing $C,c>0 $ for positive constants that only depend on $r_0$ and $n$. Their value may change from line to line.}

	\subsubsection{The unperturbed operator}

	Consider spherical coordinates $(r, \theta) \in (0,+\infty) \times S^{n-1}$ on $\R^n$. Then, writing
$d \theta^2$ for the canonical metric on the unit sphere $S^{n-1}$, the Euclidean metric is
given by
   \begin{align*}
		dx^2 = dr^2 + r^2 d\theta^2.
	\end{align*}
	For a radial function $f(r)$ there holds (see for instance~\cite[p.99]{petersen2006riemannian})
	\begin{align*}
		D^2 f = f'' dr^2 + r f' d\theta^2.
	\end{align*}
	Moreover, $df = f' dr$ and the gradient is $\nabla f = f' \p_r$ due to the simplicity of the tangent-cotangent isomorphism in this case. 
	Hence 
	\begin{align*}
		D^2 f(\nabla f, \nabla f) 
		= 
		|f'|^2 f''.
	\end{align*}
	In particular, if $\rd \in \R$ and $\rho = (r - \rd)^2 / 2$, then 
	\begin{align*}
		\rho' = r - \rd, 
		\quad
		\rho'' = 1,
	\end{align*}
	and
	\begin{align}
		D^2 \rho = dr^2 + r(r-\rd)d\theta^2 = dx^2 - \rd r d\theta^2.
	\end{align}
	This gives,
	\begin{align*}
		D^2 \rho(\nabla \rho, \nabla \rho) 
		= 
		(r - \rd)^2 
		= 2 \rho.
	\end{align*}
	Then 
	\begin{align*}
		d\phi = -t dt + d\rho,
		\quad
		\nabla \phi = t \p_t + \nabla \rho,
	\end{align*}
	and
	\begin{align}
		\label{calc_gradphi}
		(\nabla \phi, \nabla \phi) 
		=
		-t^2 + |\nabla \rho|^2 
		=
		-t^2 + (r - \rd)^2 = 2 \phi.
	\end{align}
	Moreover, 
	\begin{align*}
		D^2 \phi
		=
		- dt^2 + D^2 \rho
		=
		g - \rd r d\theta^2.
	\end{align*}
	In particular, 
	\begin{align*}
		D^2 \phi(\nabla \phi, \nabla \phi)
		=
		-t^2 + D^2 \rho(\nabla \rho, \nabla \rho) 
		= 
		-t^2 + 2 \rho
		=
		2 \phi.
	\end{align*}
	
	Let us now compute the important quantities in Lemma \ref{lem_carleman_eq}. For a vector $X \in T(\R^{1+n})$ we write $X=X^t \frac{\p}{\p_t}+X^r \frac{\p}{\p_r}+X^\theta $, $X^\theta \in T(S^{n-1})$. We have
	\begin{align}
		\label{eq_for_qp}
		Q_+(X, X)&=
		a (X, X) 
		+ 2 D^2 \ell(X, X) \\
		&=a g(X,X)+2\tau D^2\phi(X,X)=ag(X,X)+2\tau g(X,X)-2 \tau \rd r |X^\theta|^2
		\\&=
		-(a + 2 \tau) |X^t |^2 
		+ (a + 2 \tau) \left(|X^r|^2 + r^2 |X^\theta|^2 \right)
		- 2 \tau \rd r |X^\theta|^2,
	\end{align}
	and
	\begin{align}
		Q_-(\nabla \ell )&=-a(\nabla \ell, \nabla \ell) 
		+ 2 D^2 \ell (\nabla \ell, \nabla \ell)=2 \phi (-a \tau^2 + 2 \tau^3).
	\end{align}
	We need to impose 
	\begin{align*}
		(a + 2 \tau) r^2 - 2 \tau \rd r > 0, 
		\quad 
		-a + 2\tau > 0.
	\end{align*}
	We choose 
	$a = 3 \tau/2$. Then the second inequality holds, and the first one is equivalent with 
	\begin{align*}
		\frac{7 r}{4} - \rd > 0.
	\end{align*}
We define
\begin{align}
\label{def_of_rzero}
\tilde{r}_0=\frac{13}{14}\rtt    
\end{align}
and choose
	\begin{align}
		\label{choice_of_R}
		\rd = \frac{3 \rtt}{2}.
	\end{align}
With these choices, going back to~\eqref{eq_for_qp} we get that for  $r\geq \rtt_0$
	\begin{align}
		Q_+(X, X)+\frac{7}{2}\tau |X^t|^2
		&= \frac{7}{2}\tau |X^r|^2+\left(\frac{7}{2}\tau r^2-\frac{6}{2} \tau r\rtt \right )|X^{\theta} |^2\\
		&\geq  \frac{7}{2}\tau |X^r|^2+\left(\frac{7}{2}\tau r^2-\frac{6 }{ 2} \cdot \frac{14}{13}\tau r^2 \right )|X^{\theta} |^2 \\
		&\geq\frac{7}{26}\tau (|X^r|^2+ r^2 |X^\theta|^2 ),
	\end{align}
	which implies, for some $C>0$
	\begin{align}
		Q_+(X, X)+\tau |X^t|^2 \geq C^{-1} \tau |X|^2.
	\end{align}
	The length of $X$ above is defined with respect to the Euclidean metric on $\R^{1+n}$.
	We have thus obtained the following key estimates, for $v$ supported in $r\geq \rtt_0$:
	\begin{align}
		\label{pos_for_qp}
		Q_+(\nabla v, \nabla v)+\tau |\p_t v|^2 &\geq C^{-1} \tau |\nabla v |^2,
	\end{align}
	and
	\begin{align}
		\label{pos_for_qm}
		Q_-(\nabla \ell )&=-a(\nabla \ell, \nabla \ell) 
		+ 2 D^2 \ell (\nabla \ell, \nabla \ell)
		\\&\quad=
		2 \phi (-a \tau^2 + 2 \tau^3)=\phi \tau^3.
	\end{align}

	\subsubsection{Perturbation with the Gaussian weight}
	
	We need to understand how the weight $\pp$ affects the equality of Lemma~\ref{lem_carleman_eq}. Formally, we need to calculate 
    $\pp e^{\ell} \Box e^{-\ell} e^{\frac{\ve D^2_t}{2 \tau}}$. However $e^{\frac{\ve D^2_t}{2 \tau}}$ is not well defined even for compactly supported functions. We notice that the conjugated operator $\Box_\ell=e^\ell \Box e^{-\ell}$ is characterized by $e^\ell \Box=\Box_\ell e^\ell$. We want then to find $\dle$ such that $\pp e^\ell \Box w=\dle \pp e^\ell w$ or equivalently 
	$$
	\pp \Box_\ell v= \dle \pp  v,
	$$
	after the change $w=e^\ell v$. We want to commute $\pp$ with the terms in~\eqref{terms_expansion}. We remark that $\pp$ commutes with the coefficients of the Minkowski metric. For the terms depending on $t$ we need the following (see for instance~\cite[Lemma 3.12]{LL:23notes}).
	\begin{lem}
		\label{conjugation_with_mult}
		Let $u \in \mathscr{S}(\R^{1+n})$. Then
		$$
		\pp (tu)=\left(t+ \ve \frac{\p_t}{\tau}\right)\pp u.
		$$
	\end{lem}
	
	Even though we deal here with a very explicit situation where the weight function $\phi$ has a particularly simple form, it turns out that the perturbation argument is somewhat easier to write in a more abstract way.

	We use Lemma~\ref{conjugation_with_mult} and~\eqref{expansion} to find
	\begin{align}
		\label{terms_conj_with_weight}
		&\pp \Box_\ell v \\
		&= \Box  \pp v - 2 \pp (\nabla \ell, \nabla v) -  \Box \ell \pp v  + \pp (\nabla \ell, \nabla \ell) v.
	\end{align}
	We used the time independence of the metric for the first term and the fact that $\Box \ell$ is independent of $t$. We need to treat the second and fourth term in~\eqref{terms_conj_with_weight}. We write $\p_j \ell=f_1(x)+tf_0$  with $f_0$ a constant. We use Lemma~\ref{conjugation_with_mult} to get
	\begin{align*}
		\pp \p_j \ell  w&= \pp (f_1(x)+f_0)w= f_1(x)\pp w+f_0 \pp t w\\
		&=f_1(x)\pp w +f_0\left(t+\frac{\ve \p_t}{\tau}\right)\pp w  \\
		&=\left(f_1(x)+f_0t+f_0\frac{\ve \p_t}{\tau} \right) \pp w= \p_j \ell \pp w+\frac{\ve}{\tau} \p_t \p_j\ell \p_t \pp w.
	\end{align*}
	Using that the metric is time independent we find then 
	\begin{align}
		\label{conjug_fst_term}
		\pp (\nabla \ell, \nabla v)=(\nabla \ell, \nabla \pp  v)+\frac{\ve}{\tau} \left(\nabla \p_t \ell, \nabla \p_t \pp v  \right).
	\end{align}
	We proceed similarly to find for the other term
	\begin{align}
		\label{conjug_scd_term}
		\pp (\nabla \ell, \nabla \ell) v&=(\nabla \ell, \nabla \ell) \pp v+\frac{2 \ve}{\tau} \left(\nabla \p_t \ell, \nabla \ell\right) \p_t \pp v  \nonumber \\
		& \hspace{4mm}+\frac{\ve^2}{\tau^2} \left( \nabla \p_t \ell, \nabla \p_t \ell \right)\p^2_t \pp v+ \frac{\ve}{\tau} \left( \nabla \p_t \ell, \nabla \p_t \ell\right) \pp v.
	\end{align}
	Putting together~\eqref{terms_conj_with_weight},\eqref{conjug_fst_term} and~\eqref{conjug_scd_term} yields
	\begin{align}
		\dle=\Box_\ell -2A_1+A_2+A_3+A_4,
	\end{align}
	with
	\begin{align}
		A_1 v= \ve \p^2_tv, \quad 
		A_2 v = -2\ve \tau  t \p_t v , \quad 
		A_3 v =- \ve^2 \p^2_t v, \quad 
		A_4 v=-\ve \tau v.
	\end{align}
	Above we used the explicit form of $\phi$. We now have all the ingredients in order to prove the main subelliptic estimate.

	\begin{prop}
		\label{prop_subell}
		There is $C, \ve_0>0$ depending on $r_0,n$ only such that for any $\g>0$ one has
		\begin{align}
			\g \tau^3 \int  |v|^2 dtdx + \tau \int |\nabla v|^2dtdx
			\leq C \left(  \int |\dle v|^2 dtdx+  C\tau \int |\p_t v|^2 dtdx \right),
		\end{align}
		
		for all $\tau \geq \frac{C}{\gamma}$, all $\ve \leq \gamma \ve_0$ and all $v \in C^\infty_0(\R^{1+n})$ satisfying $$\supp(v) \subset  \{r\geq \rtt_0\} \cap \{ \phi \geq \g\}.$$
	\end{prop}
	
	\begin{proof}
		Recalling~\eqref{expansion} we have
		$$
		(\dle-A_4) v= (\Box+q+\sigma -2A_1+A_3)v-(L-A_2+\sigma)v,
		$$
		which gives
		\begin{align}
			|(\dle-A_4) v|^2/2&= |(\Box+q+\sigma -2A_1+A_3)v-(L-A_2+\sigma)v|^2/2 \nonumber\\
			&=|(\Box + q + \sigma-2A_1+A_3) v|^2/2 
			+ |(L-A_2 + \sigma) v|^2/2,
			\nonumber \\ & \label{expansion_gauss} \qquad 
			- \Box v(L + \sigma) v
			- (q + \sigma) v (L + \sigma) v+\rt.
		\end{align}
		The first two cross terms in~\eqref{expansion_gauss} are the same as in~\eqref{squared} and the term $\rt$ contains all the cross terms where at least one of the operators $A_j$ appears, that is
		\begin{align*}
			\rt&=2 A_1v Lv-2A_1vA_2v+2\sigma v A_1v-A_3vLv+A_3vA_2v-\sigma vA_3v\\
			&\quad+\Box v A_2 v+(q+\sigma)vA_2v.
		\end{align*}
		
		The calculations in the proof of Lemma~\ref{lem_carleman_eq} for the first two cross terms yield the following inequality,
		\begin{align}
			\int Q_-(\nabla \ell) v^2 dtdx&+\int Q_+(\nabla v)dtdx \nonumber \\ \label{int_forgauss}&\quad +\int R dtdx+\int \rt dtdx \leq \int |(\dle-A_4)v|^2 dtdx,
		\end{align}
		for $v \in C^\infty_0(\R^{1+n})$. Recall that $v \in C^\infty_0(\R^{1+n})$ is supported in $\{ \phi \geq \g \} \cap\{r\geq \rtt_0\}$. Using~\eqref{int_forgauss},~\eqref{pos_for_qm} and~\eqref{pos_for_qp} we arrive at
		\begin{align*}
			\g \tau^3 \int  |v|^2 dtdx &+ \tau \int |\nabla v|^2dtdx  +\int R dtdx+\int \rt dtdx \\
			&\leq C\int |(\dle-A_4)v|^2 dtdx+ C \tau \int |\p_t v|^2 dtdx.
		\end{align*}

		Notice that $\rd$ is fixed by~\eqref{choice_of_R}. That means in the region $[-\rd, \rd]\times \{\rtt_0 \leq r\leq \rd \}$ one has
		\begin{align}
			|\sigma|\leq C \tau, |q|\leq C \tau^2.
		\end{align}

		Since $|A_4 v |^2+|R|\leq C \tau^2 |v|^2$ by taking $\tau \geq \frac{C^\prime(r_0) }{\g}$ in the estimate above we obtain
		\begin{align}
			\label{sub_ell_with_a}
			\g \tau^3 \int  |v|^2 dtdx&+ \tau \int |\nabla v|^2dtdx  +\int \rt dtdx \nonumber \\
			&\leq C  \int |\dle v|^2 dtdx+  C\tau \int |\p_t v|^2 dtdx.
		\end{align}
		
		We need finally to absorb the term $\int \rt dt dx$ in~\eqref{sub_ell_with_a} by taking $\ve$ small enough. We start by looking at the integrals corresponding to the terms that do not contain $\sigma$. Those are $$2 A_1v Lv, 2A_1vA_2v, -A_3vLv,A_3vA_2v, \Box v A_2 v.$$ 
		We observe that they are of the form $\ve^{j} \tau (Av,Bv)_{L^2(M)}, j \in\{1,2\}$ with $A$ a self-adjoint operator of order $2$ and $B$ of order $1$ satisfying $B^*=-B+C$ with $C$ of order $0$. Note as well that both $A,B$ are independent of $\tau$ and $\ve$. Up to considering the real and imaginary part of $v$ we can assume that it is real valued. We have then
		$$
		2(Av,Bv)_{L^2(M)}= (Av,Bv)_{L^2(M)}+(Bv,Av)_{L^2(M)}=(([A,B]+CA)v,v)_{L^2(M)},
		$$
		and $[A,B]+CA$ is a differential operator of order $2$. As a consequence after an integration by parts we can bound
		$$
		\ve^j \tau |(Av,Bv)_{L^2(M)}|\leq C \ve \tau \left(\int |v|^2 dtdx+ \int |\nabla v |^2 dtdx\right),
		$$
		and the corresponding term can be absorbed in~\eqref{sub_ell_with_a} by taking $\tau\geq \tau_0$ large enough. We have for the term corresponding to $\sigma vA_3v$  
		\begin{align*}
			\left|\int \sigma v A_3 v dt dx \right|  = \ve^2  \left| \int \sigma v \p_t^2 v dtdx  \right| \leq C \ve^2 \tau \left (\int |\p_t v|^2 dtdx+\int | v|^2 dtdx \right),
		\end{align*}
		after an integration by parts. We find similarly for the term $2\sigma v A_1v$
		\begin{align*}
			\left|2 \int \sigma v A_1 v dt dx \right| \leq C \ve \tau \left (\int |\p_t v|^2 dtdx+\int | v|^2 dtdx \right),
		\end{align*}
      and these terms can be as well absorbed by taking $\tau$ large enough. For the last term we integrate by parts and find
		\begin{align}
        \label{term_bad}
		\left| \int (q+\sigma)vA_2v dtdx \right | \leq C \ve \tau^3 \int |v|^2.    
		\end{align}
Now we can still absorb this term but it does not suffice to take $\tau$ large since the power of $\tau$ in $\tau^3 \int |v|^2$ is the critical one. To absorb this term is necessary that we take $\ve \leq  \g \ve_0 $ small for some $\ve_0$ depending on $r_0,n$.
        
        We finally obtain that all terms coming from $\int \rt dtdx$ can indeed be absorbed in the right hand side of~\eqref{sub_ell_with_a} by taking $\ve \leq \g \ve_0$ small enough. This completes the proof of the proposition.
	\end{proof}

\begin{remark}
\label{rem_espilon_depends_ongamma}
It follows from the proof above that in order to absorb the additional terms coming from the Gaussian weight one has to choose $\ve$ small \textit{depending on $\g$}. This will turn out to play a role in the dependency of the constants on $\delta$, see Remark~\ref{remark_dependancy on g}. In particular one would hope to choose $\ve$ small depending on $r_0,n$ but independent of $\g$. Carrying out explicit calculations one sees that this can be achieved for all but one of the terms coming from $A_j$, $j=1,2,3$. It is the term $(q+\sigma)vA_2v$ that forces the coupling between $\g$ and $\ve$.	

More precisely, we have thanks to~\eqref{calc_gradphi}
 \begin{align}
   (q+\sigma)=G(\nabla \ell)+a=\tau^2(\nabla \phi,\nabla \phi)+3\tau/2=2\tau^2 \phi+3\tau/2.
 \end{align}
We look at the left hand side of~\eqref{term_bad}. Integrating by parts we see that $ \int  \frac{3 \tau}{2} v A_2 v$ does not force the coupling so we focus on
\begin{align}
    \int  2 \tau^2 \phi v A_2 v dtdx&= 2 \tau^3 \ve \left( \int (\p_t \phi)t v^2 dtdx+ \int \phi v^2 dtdx \right) \\
    &=2 \tau^3 \ve \left(  \int \phi v^2 dtdx-\int t^2 v^2 dtdx \right).
\end{align}
Keeping track of signs, we see that the only problematic term is $-2 \tau^3 \ve \int t^2 v^2 dtdx$, since it pops out with the bad sign.

\end{remark}

	\subsection{Using the Gaussian weight}
	
We are ready to prove the global Carleman estimate.

	\begin{thm}
		\label{thm_explicit_carleman}
		There is $C>0, \a=\a(r_0)>0$ and $\ve=\ve(r_0)>0$ depending only on $r_0,n$ such that for $0<\g \leq 1$ one has
		\begin{align}
			\tau^3 \norm{\pp e^{\tau \phi}u}{L^2}{2}+\tau \norm{\nabla \pp e^{\tau \phi}u}{L^2}{2}&\leq \frac{C}{\g} \left( \norm{\pp e^{\tau \phi} \Box u }{L^2(\{\phi > \g/2\})}{2}+e^{-\a \g^2 \tau} \norm{e^{\tau \phi}u}{H^1}{2}\right),
		\end{align}
		for all $\tau  \geq \frac{C}{\g^8}$, $u \in C^\infty_0(\{ \phi \geq \g \})$ satisfying $\supp(u) \subset \{r \geq \rtt_0\}$ and $\ve=\g \ve(r_0)$.
	\end{thm}

	\begin{proof}

        To alleviate notation we write $\Omega_\g=\{ \phi > \g/2\}$. We fix $\ve=\gamma \ve_0$ as given by Proposition~\ref{prop_subell}. We define $v= \pp e^{\tau \phi}  u$. Notice that since $\pp$ is not local we can not directly apply the subelliptic estimate of Proposition~\ref{prop_subell} to $v$.
		
		We consider two cut-offs $\chi_1 \in C^\infty_0(\{\phi \geq \g\})$ such that $\chi_1=1$ on $\{\phi \geq 2\g\}$ and $\chi_2 \in C^\infty_0(\{\phi \geq 4\g\})$ such that $\chi_2=1$ on $\{\phi \geq 8\g\}$. Then the function 
		$$
		\chi_1 v= \chi_1 \pp e^{\tau \phi}  u= \chi_1 \pp e^{\tau \phi} \chi_2 u,
		$$
		satisfies $\supp( \chi_1 v) \subset  \{r\geq \rtt_0\} \cap \{ \phi \geq \gamma\}$. We write 
		\begin{align}
			\label{estim_1}
			\norm{v}{H^1_\tau}{} \leq  \norm{\chi_1 v}{H^1_\tau}{}+ \norm{(1-\chi_1)v}{H^1_\tau}{}.
		\end{align}
		We can estimate the second term with the help of Lemma~\ref{lem_almostlocal} and by noticing that $\dist(\supp(1-\chi_1),\chi_2))\geq c\g$ as follows:
		\begin{align}
			\norm{(1-\chi_1)v}{H^1_\tau}{} \leq \frac{C}{\g^2}e^{-c\frac{\gamma^2 \tau}{2\ve}} \norm{e^{\tau \phi }u}{H^1_\tau}{}.
		\end{align}
		We apply Proposition~\ref{prop_subell} for the other term to find, for $\g < 1 $
		\begin{align}
			\label{appl_of_subell}
			&\g \left(\tau^3 \int  |\chi_1 v|^2 dtdx + \tau \int |\nabla \chi_1 v|^2dtdx \right)\\
			&\leq C  \int |\dle (\chi_1 v)|^2 dtdx+  C\tau \int |\p_t (\chi_1 v)|^2 dtdx,
		\end{align}
		and we need to bound the two terms on the right hand side of~\eqref{appl_of_subell}.
		We start by writing 
		\begin{align}
			\dle (\chi_1v)=\chi_1 \dle v +[\dle,\chi_1]v=\chi_1 \dle v +[\dle,\chi_1]\pp e^{\tau \phi} \chi_2 u,
		\end{align}
		which implies thanks to the support properties of $\chi_1, \chi_2$ and Lemma~\ref{lem_almostlocal} that
		\begin{align}
			\label{estim_2}
			\int |\dle (\chi_1 v)|^2 dtdx \leq C \int_{\Omega_\g} |\dle v|^2 dtdx+\frac{C}{\g^4}e^{-c\frac{\gamma^2 \tau}{\ve}} \norm{e^{\tau \phi }u}{H^1}{2}.
		\end{align}
		For the other term we write as well $\p_t (\chi_1v)=\chi_1 \p_t v+[\p_t,\chi_1]\pp e^{\tau \phi} \chi_2 u$ which implies again thanks to Lemma~\ref{lem_almostlocal}
		\begin{align}
			\label{estim_3}
			\int |\p_t (\chi_1 v)|^2 dtdx \leq C \int_{\Omega_\g} |\p_t v|^2 dtdx+ \frac{C}{\g^4}e^{-c\frac{\gamma^2 \tau}{\ve}} \norm{e^{\tau \phi }u}{H^1}{2}.
		\end{align}
		To get a bound for $\norm{D_t v}{L^2}{}$ we work on the Fourier domain (with respect to the time variable $t$) and distinguish between frequencies smaller or bigger than $\sigma \tau$ with $\sigma>0$ to be chosen. We have
		\begin{align*}
			\norm{D_t v}{L^2}{}&\leq \norm{1_{|D_t|\leq \sigma \tau}D_t v}{L^2}{}+\norm{1_{|D_t|\geq \sigma \tau}D_t v}{L^2}{}\\
			&=\norm{1_{|D_t|\leq \sigma \tau}D_t v}{L^2}{}+\norm{1_{|D_t|\geq \sigma \tau} D_t e^{-\frac{\ve D_t^2}{2 \tau}} e^{\tau \phi }u}{}{}\\
			&\leq \norm{1_{|D_t|\leq \sigma \tau}D_t v}{L^2}{}+  \underset{\xi_t \geq \sigma \tau}{\textnormal{max}} \left( \xi_t e^{-\frac{\ve \xi^2_t}{2 \tau}} \right)\norm{e^{\tau \phi}u}{L^2}{}.
		\end{align*}
		Now the function $\R^+  \ni  s \mapsto s e^{-\ve\frac{s^2}{2\tau}}$ reaches its maximum at $s=\sqrt{\frac{\tau}{\ve}}$ and is decreasing on $[\sqrt{\frac{\tau}{\ve}}, \infty)$. As a consequence, if $\tau \geq \frac{1}{\sigma^2 \ve}$, one has $\underset{\xi_t \geq \sigma \tau}{\max}(\xi_t e^{-\ve\frac{\xi_t^2}{2\tau^3}})=\sigma \tau e^{- \ve \frac{\sigma^2 \tau}{2}}$.
		We obtain therefore, for $\tau \geq \tau_0 \geq \frac{1}{\sigma^2\ve}$,
		\begin{align}
			\label{estim_gauss_sigma}
			\tau \norm{\p_t v}{L^2}{2} \leq \sigma^2 \tau^3 \norm{v}{L^2}{2}+\sigma^2 \tau^3e^{-\ve \sigma^2 \tau} \norm{e^{\tau \phi}u}{L^2}{2}.
		\end{align}
		We collect estimates~\eqref{estim_1},~\eqref{appl_of_subell},~\eqref{estim_2}~\eqref{estim_3} and~\eqref{estim_gauss_sigma} to find
		\begin{align}
			\tau^3 \norm{v}{L^2}{2}+\tau \norm{\nabla v}{L^2}{2}&\leq \frac{C}{\g}\left(\norm{\dle v}{L^2}{2}+\frac{C}{\g^4}e^{-c\frac{\gamma^2 \tau}{\ve}} \norm{e^{\tau \phi }u}{H^1}{2}\right)\\
			&\quad+\frac{C}{\g}\left(\sigma^2 \tau^3 \norm{v}{L^2}{2}+\sigma^2 \tau^3e^{-\ve \sigma^2 \tau} \norm{e^{\tau \phi}u}{L^2}{2}\right)\\
			&\quad+\frac{C}{\g^4}\left(\tau^3e^{-c\frac{\gamma^2 \tau}{\ve}} \norm{e^{\tau \phi }u}{L^2}{2}+\tau e^{-c\frac{\gamma^2 \tau}{\ve}} \norm{\nabla (e^{\tau \phi }u)}{L^2}{2}\right).
		\end{align}
		Recall that $\ve=\g \ve_0$ with $\ve_0$ depending only on $r_0,n$. As a consequence for $\tau\geq \frac{C}{\g^8}$ we have $\frac{\tau^3}{\g^4}e^{-c\frac{\gamma^2 \tau}{\ve}} \leq Ce^{-c^\prime \g \tau}$. We obtain then:
		\begin{align}
			\tau^3 \norm{v}{L^2}{2}+\tau \norm{\nabla v}{L^2}{2}&\leq \frac{C}{\g}\norm{\dle v}{L^2(\Omega_\g)}{2}+\frac{C}{\g }e^{-c \g \tau} \norm{e^{\tau \phi }u}{H^1}{2}\\
			&\quad+\frac{C}{\g}\left(\sigma^2 \tau^3 \norm{v}{L^2}{2}+\sigma^2 \tau^3e^{-\ve \sigma^2 \tau} \norm{e^{\tau \phi}u}{L^2}{2}\right).
		\end{align}
		We now choose $\sigma^2=\frac{\g}{2C}$ and recall the definition of $\dle$ to arrive at 
		\begin{align}
			\frac{\tau^3}{2} \norm{v}{L^2}{2}+\tau \norm{\nabla v}{L^2}{2}\leq \frac{C}{\g} \norm{\pp e^{\tau \phi} \Box u }{L^2(\Omega_\g)}{2}+\frac{C}{\g}e^{-c\g \tau} \norm{e^{\tau \phi }u}{H^1}{2}+\frac{\tau^3}{2}e^{-\ve \frac{\g}{2C}  \tau} \norm{e^{\tau \phi}u}{L^2}{2}.    
		\end{align}
 Recalling $e^{-\ve \frac{\g}{2C}  \tau}=e^{-\g^2 c \tau }$ we obtain, for $\tau \geq \frac{C}{\g^8}$
		\begin{align}
			\label{estim_before_gamma8}
			\tau^3 \norm{v}{L^2}{2}+\tau \norm{\nabla v}{L^2}{2}&\leq \frac{C}{\g} \left( \norm{\pp e^{\tau \phi} \Box u }{L^2(\Omega_\g)}{2}+e^{-c \g\tau} \norm{e^{\tau \phi }u}{H^1}{2}\right)\\&\quad+C\tau^3e^{-\g^2 c \tau} \norm{e^{\tau \phi}u}{L^2}{2}.
		\end{align}
		We remark that for $\tau \geq \frac{C}{\g^8}$ with $C$ large enough depending only on $r_0,n$ one has
		\begin{align}
			\tau^3e^{-\g^2 c \tau} \leq C e^{\frac{-\g^2 c \tau}{2}},
		\end{align}
		which holds since for some $C^\prime>0$
		\begin{align}
			C^\prime e^{\frac{\g^2 c \tau}{2}} \geq (\g^2 \tau)^4=\g^8 \tau \cdot \tau^3\geq C \tau^3.
		\end{align}
		Combining this with~\eqref{estim_before_gamma8} finally yields that there is $C$ depending only on $r_0,n$ such that for all $\tau \geq \frac{C}{\g^8}$
		\begin{align}
			\tau^3 \norm{v}{L^2}{2}+\tau \norm{\nabla v}{L^2}{2}&\leq \frac{C}{\g} \left( \norm{\pp e^{\tau \phi} \Box u }{L^2(\Omega_\g)}{2}+e^{-c \g \tau} \norm{e^{\tau \phi }u}{H^1}{2}\right)\\&\quad+Ce^{-c \g^2  \tau} \norm{e^{\tau \phi}u}{L^2}{2},
		\end{align}
		which implies in particular that for all $\tau \geq \frac{C}{\g^8}$
		\begin{align}
			\tau^3 \norm{v}{L^2}{2}+\tau \norm{\nabla v}{L^2}{2}&\leq \frac{C}{\g} \left( \norm{\pp e^{\tau \phi} \Box u }{L^2(\Omega_\g)}{2}+e^{-\g^2 c \tau} \norm{e^{\tau \phi}u}{H^1}{2}\right).
		\end{align}
		This proves the theorem, up to replacing $\g$ by $\frac{\g}{8}$ and renaming the constants.
	\end{proof}
	
	\bigskip
	
Several remarks are in order.	
	
	\begin{remark}
		\label{remark_gamma_small}
		The assumption $\g\leq 1$ made in Theorem~\ref{thm_explicit_carleman} allows to focus on the region of interest, that is close to the characteristic cone $\{\phi=0\}$. If $\g\geq 1$ then one can have an estimate with the constant $C/\g$ replaced by some constant depending only on $r_0,n$ (and which does not blow up as $\g$ approaches $0$), but this is of no interest to us.
	\end{remark}

\begin{remark}
	\label{remark_dependancy on g}
In order to control the term $\norm{D_t v}{L^2}{}$ we had to distinguish between frequencies smaller or larger than $\sigma \tau$ and then choose $\sigma$ small depending on $\g$. As it turns out, the exponentially small term $e^{-\a \g^2 \tau} \norm{e^{\tau \phi}u}{H^1}{2}$ will be the dominant one in the unique continuation argument, see Proposition~\ref{prop_qual_uniq}. The quantity $\a \g^2$ roughly corresponds to how much one can propagate uniqueness. The factor $\g^2$ comes from the forced choices $\ve=c\g$ and $\sigma^2=c^\prime \g$. One could try to choose $\ve$ or $\sigma$ independent of $\g$ in order to improve the unique continuation result, in the sense of advancing further in the level sets of $\phi$. However, examination of the proof shows that $\ve$ and $\sigma$ cannot be chosen independent of $\g$. See as well Remark~\ref{rem_espilon_depends_ongamma}.
	
\end{remark}

	\begin{remark}
    \label{rem_potential}
		Consider a time independent potential $q=q(x)$ satisfying $\norm{q}{L^\infty}{}\leq M$. Then using the fact that it commutes with the Gaussian weight we can write
		\begin{align}
			\norm{\pp e^{\tau \phi} (\Box+q)u}{L^2}{} &\geq  \norm{\pp e^{\tau \phi}  \Box u}{L^2}{}- \norm{\pp e^{\tau \phi} qu}{L^2}{}\\
			&=\norm{\pp e^{\tau \phi}  \Box u}{L^2}{}- \norm{q \pp e^{\tau \phi} u}{L^2}{} \\
			&\geq \norm{\pp e^{\tau \phi}  \Box u}{L^2}{}-M \norm{\pp e^{\tau \phi} u}{L^2}{},
		\end{align}
		which implies that Theorem~\ref{thm_explicit_carleman} remains valid for $\Box+q$ up to taking $\tau \geq \frac{C(r_0,M)}{\g}$.
	\end{remark}

       \begin{remark}
		The same estimate is also valid for a weight $\psi=\phi+c$ for a constant $c$. This follows form the fact that the key quantities $Q_\pm$ are the same for $\phi$ and $\psi$.     
	\end{remark}

	\subsection{Qualitative unique continuation as a consequence: an example}

The following proposition is a qualitative unique continuation result which we prove here in order to see how one can use the Carleman estimate of Theorem~\ref{thm_explicit_carleman} to prove uniqueness. This serves as an illustration of how we propagate uniqueness, but is not used in what follows. We recall the definition of the cylinder $\C$ in~\eqref{def_cylinder} (here $R$ is replaced by $\rtt$) and $\rtt_0$ in~\eqref{def_of_rzero}.
	\begin{prop}
		\label{prop_qual_uniq}
		There is a constant $a>0$ depending on $r_0,n$ only such that the following unique continuation property holds. Consider $u$ satisfying $\Box u =0$ on $\R^{1+n}$ and $u=0$ on $\C$. Let $\delta>0$. Then for any $\g>0$ with $\g- a \delta^2 \geq \delta$ we have that $u=0$ on $\{\phi \geq \g \}$ implies $u=0$ on $\{\phi \geq \g- a \delta^2 \}$.
	\end{prop}

	\begin{proof}
		
		Consider $u$ satisfying $\Box u =0$ on $\R^{1+n}$ and $u=0$ on $\C$. Notice that by finite speed of propagation this implies that $u=0$  in $\{\phi \geq 0\} \cap \{r \leq \rtt\}$. Let $\chi \in C^\infty_0(\{\phi \geq \delta/2\})$ such that $\chi_1=1$ on $\{\phi \geq \delta\}$. We have in particular that $\supp(\chi u) \subset \{\phi \geq \delta/2\} \cap \{r \geq \rtt_0\} $. We apply the estimate of Theorem~\ref{thm_explicit_carleman} to $\chi u$ and downgrade it to find that for $\tau \geq \frac{C}{\delta^{8}}$ one has
		\begin{align}
			\norm{\pp e^{\tau \phi}(\chi u)}{L^2}{2}&\leq \frac{C}{\delta} \left( \norm{\pp e^{\tau \phi} \Box (\chi u) }{L^2}{2}+e^{-\a \delta^2 \tau} \norm{e^{\tau \phi}\chi u}{H^1}{2}\right).
		\end{align}
		We need to control the right hand side of the estimate above. We write
		\begin{align}
			\Box(\chi u)=\chi \Box u+ [\Box, \chi] u= [\Box, \chi] u,
		\end{align}
		and use the fact that the commutator is supported on $\supp(\nabla \chi)$ where $\phi \leq  \delta$. This gives
		\begin{align}
			\norm{\pp e^{\tau \phi} \Box (\chi u) }{L^2}{2} \leq \norm{ e^{\tau \phi} \Box (\chi u) }{L^2}{2} \leq e^{2 \delta \tau } \norm{u}{H^1}{2}.
		\end{align}
		Recalling that $\chi u$ is supported on $\{\phi \leq \g \}$ and defining $a=\a/4$ we have as well
		\begin{align}
			e^{-\a \delta^2 \tau} \norm{e^{\tau \phi}\chi u}{H^1}{2} \leq \tau^2 e^{(2\g-\a \delta^2)\tau}\norm{u}{H^1}{2} \leq C e^{(2\g-2a \delta^2)\tau}\norm{u}{H^1}{2}.
		\end{align}
Hence
\begin{align}
	\norm{\pp e^{\tau \phi}(\chi u)}{L^2}{2} \leq \frac{C}{\delta} \left(e^{2 \delta \tau}+ e^{(2\g-2a \delta^2)\tau} \right) \norm{u}{H^1}{2}.
	\end{align}	
Using the assumption $2\g -2 a \delta^2 \geq 2\delta$   we obtain therefore
		\begin{align}
			\norm{\pp e^{\tau \phi}  (\chi u) }{L^2}{} \leq \frac{C}{\sqrt{\delta}} e^{(2\g-2 a\delta^2)/2\tau}\norm{u}{H^1}{}.
		\end{align}
Applying Lemma~\ref{lemme d analyse harmonique} this gives $\chi u =0$ on $\{\phi \geq \g- a \delta^2 \}$. Since $\chi=1$ on $\{\phi \geq \delta \}$ the assumption $\g- a \delta^2 \geq \delta$ implies that $u=0$ on the set $\{\phi \geq \g- a \delta^2 \}$ and hence the statement.
	\end{proof}
	
\begin{remark}
\label{rem_numberofiterations}
Proposition~\ref{prop_qual_uniq} above says that even though we start from a \textit{global} Carleman estimate, as far as the unique continuation procedure is concerned the exponentially small remainder term $e^{-\a \delta^2 \tau} \norm{e^{\tau \phi}\chi u}{H^1}{2}$ becomes in fact the dominant one. We can only propagate information from $\{\phi \geq \gamma\}$ to $\{\phi \geq \g- c \delta^2\}$ for a constant $c$ depending only on $r_0,n$. As a consequence, in order to prove uniqueness up to the wave cone (which is the optimal domain) one needs to iterate this procedure. This can only be achieved with polynomial number of iterations in $1/\delta$ but not with a logarithmic number of iterations. If the latter was possible, we would get improvement in $\ct(\delta)$.
\end{remark}

\section{The local explicit quantitative estimate}
\label{sec_local}
	
	We will now use the procedure in~\cite{Laurent_2018} in order to prove a  \textit{quantitative} version of Proposition~\ref{prop_qual_uniq}. In order to get an estimate that can be iterated one has to only propagate \textit{low frequency} information with respect to time. Let $m(t)$ be a smooth radial function, compactly supported in $|t|<1$ such that $m(t)=1$ for $|t|<3/4$. We shall denote by $M^\mu$ the Fourier multiplier defined $M^\mu=m(\frac{D_t}{\mu})$, that is
	$$
	M^\mu u (t,x)=\mathcal{F}_t^{-1}\left(m\left(\frac{\xi_t}{\mu}\right) \mathcal{F}_t (u)(\xi_t,x)\right)(t).
	$$
	Therefore the upper index $\mu$ translates to an operator that localizes to times frequencies smaller than $\mu$. We shall also use a regularization operator. Given $f \in L^\infty(\R^{1+n})$ we set
	$$
	f_\lambda(t,x):=e^{-\frac{D_t^2}{\lambda}}f=\left(\frac{\lambda}{4 \pi}  \right)^{\frac{1}{2}}\int_{\R} f(s,x)e^{-\frac{\lambda}{4}|t-s|^2}ds.
	$$
	That is, the lower index $\lambda$ produces an analytic function with respect to the time variable. We will need also the combination of the two procedures above. Given $\lambda, \mu >0$, we write $M_\lambda^\mu$ for the Fourier multiplier defined by $M_\lambda^\mu=m_\lambda (\frac{D_t}{\mu})$ or more precisely:
	$$
	(M_\lambda^\mu u)(t,x)=\mathcal{F}_t^{-1}\left(m_\lambda\left(\frac{\xi_t}{\mu}\right)\mathcal{F}_t(u) (\xi_t,x)\right)(t).
	$$
	That is, we first regularize and then localize. In the sequel we shall use the localization and regularization parameters $\lambda, \mu >0$ and we will suppose that $\lambda \sim \mu$, that is 
	$$
	1/\tilde{C} \mu \leq \lambda \leq \tilde{C} \mu,
	$$
	for some $\tilde{C}>0$.
	
	\bigskip

	We consider a small parameter $\delta>0$ which we think of as the square of the distance to the optimal cone and a larger parameter $\g>0$ which plays the role of $\g$ in Proposition~\ref{prop_qual_uniq}. That is, we are interested in propagating information close to the hypersurface $\{\phi=\g\}$. We define 
	\begin{align}
		\label{def_of_zeta}
		\zeta=\frac{\a \delta^2}{16},
	\end{align}
	where the constant $\a$ is the one in Theorem~\ref{thm_explicit_carleman}. We consider then $\sigma \in  C^\infty_0(\{\g-\frac{\zeta}{4}\leq \phi  \leq \g+\frac{\zeta}{4}\})$ such that $\sigma=1$ in a neighborhood of $\{\g-\frac{\zeta}{8}\leq \phi  \leq \g+\frac{\zeta}{8}\}$. We write finally $\Omega_\delta= \{\phi > \delta\}$.

	\begin{thm}
		\label{thm local quant estimate}
 There exist $C,N>0$ depending on $r_0,n$ only such that for  any $\theta \in C^\infty_0(\R^{1+n})$ with $\theta(x)=1 $ on a neighborhood of $\{\phi > \g+\frac{\zeta}{8} \} $ we have that for all $ \kappa, \alpha>0$ there exist $\kp \sim  \min(\delta^{N}, \kappa \delta^{N},  \alpha \delta^{N})$ and $\beta= \min(\delta^{N}, \kappa \delta^{N},  \alpha \delta^{N})$ such that for all $\mu \geq \frac{C}{ \delta^8 \beta}$ and $u \in C^\infty_0(\R^{1+n})$ with $\supp (u) \subset \{r\geq \rtt_0\}$, one has:
	\begin{equation*}
			\norm{M^{\beta \mu}_{ \mu} \sigma_{\mu }u}{H^1}{} \leq \frac{C }{\delta^{N}} e^{\kappa  \mu}\left(\norm{M^{\alpha \mu}_{\mu }\theta_{\mu }u}{H^1}{}+\norm{\Box u}{L^2(\Omega_\delta)}{}\right)+\frac{C}{\delta^{N}}e^{-\kp \mu } \norm{u}{H^1}{}.
		\end{equation*}
    \end{thm}

   The proof of Theorem~\ref{thm local quant estimate} will be done in two main steps following Sections 3.2 and 3.3 of~\cite{Laurent_2018}. Here, we have to keep trace of the dependency of all the constants with respect to the parameter $\delta>0$. 
	
	We consider the following cut-offs that we need in order to apply the Carleman estimate of Theorem~\ref{thm_explicit_carleman}. Let $\chi \in C^\infty_0((-8\zeta,\zeta))$ such that  $\chi=1$ in $[-7\zeta, \zeta/2]$ and $\tchi \in C^\infty$ supported in $(-\infty,2\zeta))$ such that  $\tchi=1$ in $(-\infty, 3/2\zeta)$. Let $\sigma_2 \in C^\infty_0(\{ \phi \geq \delta\})$ with $\sigma_2=1$ in $\{\phi \geq 2 \delta\}$ and $\sigma_1 \in C^\infty_0(\{ \phi \geq 2 \delta\})$ with $\sigma_1=1$ in $\{\phi \geq 3 \delta\}$.

	We will use $$\psi= \phi-\g,$$
	as a weight function and apply the Carleman estimate of Theorem~\ref{thm_explicit_carleman} to $\sigma_2 \sigma_{1, \lambda}\tchi(\psi)\chil (\psi) u$.

	We start by giving an estimate for the term $ \norm{\pp e^{\tau \psi} \Box \sigma_2 \sigma_{1, \lambda}\tchi(\psi)\chil (\psi) u }{L^2}{}$ which appears in the right hand side of Theorem~\ref{thm_explicit_carleman}. We have
	\begin{align}
		&\norm{\pp e^{\tau \psi} \Box \sigma_2 \sigma_{1, \lambda}\tchi(\psi)\chil (\psi) u }{L^2}{} \\
		&\leq  \norm{\pp e^{\tau \psi}  \sigma_2 \sigma_{1, \lambda}\tchi(\psi)\chil (\psi)  \Box u }{L^2}{}+ \norm{\pp e^{\tau \psi}[\sigma_2 \sigma_{1, \lambda}\tchi(\psi)\chil (\psi) , \Box ]u }{L^2}{} \\
		&\leq   \norm{e^{\tau \psi}  \sigma_2 \sigma_{1, \lambda}\tchi(\psi)\chil (\psi)  \Box u }{L^2}{}+ \norm{\pp e^{\tau \psi}[\sigma_2 \sigma_{1, \lambda}\tchi(\psi)\chil (\psi) , \Box ]u }{L^2}{} \\
		&\label{estim_for_carl}\leq C \mu^{1/2}e^{\tau^2/\mu}e^{\zeta \tau } \norm{\Box u}{L^2(\Omega_\delta)}{}+\norm{\pp e^{\tau \psi}[\sigma_2 \sigma_{1, \lambda}\tchi(\psi)\chil (\psi) , \Box ]u }{L^2}{},
	\end{align}
	thanks to Lemma~\ref{lemma 2.13 from ll}.
	We need now to bound the commutator. This is the content of the following lemma.
	
	\begin{lem}
		\label{lem_commut}
		With the above notations, there exist $C,c,N>0$ depending on $r_0,n$ only such that for  any $\theta \in C^\infty_0(\R^{1+n})$ with $\theta(x)=1 $ on a neighborhood of $\{\phi > \g \} $ one has
		\begin{align}
			&\norm{\pp e^{\tau \psi}[\sigma_2 \sigma_{1, \lambda}\tchi(\psi)\chil (\psi) , \Box ]u }{L^2}{} \\
			& \leq \frac{C}{\delta^4} e^{2 \zeta \tau} \norm{M^{2 \mu}_\lambda \theta_\lambda u}{H^1}{}+\frac{C}{\delta^N}\mu^{1/2}\tau^N \left(e^{-8 \zeta \tau}+e^{-\frac{\ve \mu^2}{8 \tau}}+e^{-c\delta^4 \mu}e^{\zeta \tau} \right) e^{C \tau^2/ \mu}e^{\zeta \tau} \norm{u}{H^1}{},  
		\end{align}
		for all $u \in C^\infty_0(\R^{1+n})$, $\mu \geq 1$, $\lambda \sim \mu$ and $\tau \geq 1$.
	\end{lem}
	\begin{proof}
		By the Leibniz rule we can write 
		\begin{multline*}
			\partial^\alpha(\sigma_2 \sigma_{1, \lambda}\tchi(\psi)\chil (\psi)u)\\
			=\sum_{\alpha_1+\alpha_2+\alpha_3+\alpha_4+\alpha_5=\alpha}C_{(\alpha_i)}\partial^{\alpha_1}(\chi_{\lambda}(\psi))\partial^{\alpha_2}(\sigma_{2})\partial^{\alpha_3}(\sigma_{1, \lambda})\partial^{\alpha_4}(\tchi(\psi))\partial^{\alpha_5}u.
		\end{multline*}
		
		The commutator $[\Box, \sigma_2 \sigma_{1, \lambda}\tchi(\psi)\chil (\psi) ]$ can be split as a sum of differential operators of order one as follows:
		\begin{equation*}
			\sum_{|\alpha|\leq 2} [\partial^\alpha,\sigma_{2}\sigma_{1, \lambda}\tchi_{\delta}(\psi)\chi_{\delta,\lambda}(\psi)]
			=B_1+B_2+B_3+B_4,
		\end{equation*}
		where:
		
		\begin{enumerate}
			\item $B_1$ contains the terms with $\alpha_1 \neq 0$ and $\alpha_2=\alpha_4=0$;
			
			\item $B_2$ contains some terms with $\alpha_2 \neq 0$;
			
			\item $B_3$ contains the terms with $\alpha_3 \neq0$ and $\alpha_1=\alpha_2=\alpha_4=0$;
			
			\item $B_4$ contains some terms with $\alpha_4 \neq 0$.
		\end{enumerate}
		
		Using the support properties of $\chi$ we observe that $\partial^{\alpha_1}\chil (\psi)$ with $\alpha_1 \neq 0$ can be decomposed in some generic terms $\chil^{\pm}$ which are such that $\chi^{-}$ is supported in $[-8\zeta, -7\zeta]$ and some terms  $\chi_{ \lambda}^{+}$ such that $\chi^{+}$ is supported in $[\zeta/2, \zeta]$.

		\begin{enumerate}
			\item This allows to decompose $B_1$ as a sum of generic terms of the form
			\begin{equation}
				\label{def of B+-}
				B_{\pm}=b_{\pm}\partial^\gamma=f \sigma_{2}\partial^{\beta}(\sigma_{1,\lambda})\chi_{ \lambda}^{\pm}\tchi(\psi)\partial^\gamma,	
			\end{equation}

			where $|\gamma| \leq 1$, $f \in C^\infty_0(\R^{1+n})$ contains some derivatives of $\psi$ and is analytic in $t$. 
			
			\item $B_2$ consists of terms where there is at least one derivative on $\sigma_{2}$ and contains terms of the form
			$$
			\tilde{b}\partial^{\beta}(\sigma_{1, \lambda})(\chi_{\lambda})^{(k)}(\psi)\tchi(\psi)\partial^\gamma,
			$$
			where $k, |\beta|, |\gamma| \leq 1$.
			
			\item $B_3$ consists of terms where there is at least one derivative on $\sigma_{1, \lambda}$ and none on $\chi_{\lambda}(\psi), \tchi(\psi)$ and $\sigma_2$. These are terms of the form 
			$$
			f \sigma_{2}\partial^{\beta}(\sigma_{1, \lambda})\chi_{\lambda}(\psi)\tchi(\psi)\partial^\gamma.
			$$
			
			\item $B_4$ consists of terms where there is at least one derivative on $\tchi(\psi)$ and contains terms of the form
			$$
			\tilde{b}\partial^{\beta}(\sigma_{1,\lambda})(\chi_{\lambda})^{(k)}(\psi)\partial^\gamma,
			$$
			where $k, |\beta|, |\gamma|\leq 1 $.

		\end{enumerate}
		
		We notice that derivatives of $\tchi(\psi)$ and $\chi(\psi)$ can be bounded from above by $C/\delta^2$ with $C$ depending on $r_0,n$ only. We now estimate each of the generic terms considered above.
		
		\bigskip

		\noindent \textbf{Estimating $B_-$} . We use Lemma~\ref{lemma 2.13 from ll} applied to $\chi^{-}$ to find:
		\begin{equation*}
			\norm{ \pp e^{\tau \psi} B_- u}{L^2}{} \leq \norm{e^{\tau \psi}B_-}{L^2}{}\leq \frac{C}{\delta^2}\lambda^{1/2}e^{-7 \zeta \tau }e^{\frac{\tau^2}{\lambda}}\norm{u}{H^1}{}\leq \frac{C}{\delta^2}\mu^{1/2}e^{-7 \zeta \tau }e^{C \frac{\tau^2}{\mu}}\norm{u}{H^1}{}.
		\end{equation*}

		\noindent\textbf{Estimating $B_2$.} We use Lemma~\ref{lemma 2.3 from ll} applied to $\tilde{b}$ and $\partial^{\beta}(\sigma_{1, \lambda})$ which have supports away from each other. We apply then Lemma~\ref{lemma 2.13 from ll} to $\chi^{(k)}$ using its support properties to find:
		\begin{align*}
			\norm{e^{-\ve \frac{D_t^2}{2  \tau}}e^{\tau \psi} B_2 u}{L^2}{}&\leq\norm{e^{\tau \psi}B_2 u}{L^2}{}\leq \frac{C}{\delta^2} \mu^{1/2}e^{\zeta \tau}e^{C \frac{\tau^2}{\mu}}e^{-c \delta^4 \mu}\norm{u}{H^1}{}.
		\end{align*}
		
		\noindent \textbf{Estimating $B_4$.} We use Lemma~\ref{lemma 2.3 from ll} applied to $(\chi_{\lambda})^{(k)}(\psi)$ and $\mathds{1}_{[3\delta/2,2\delta]}$ to find thanks to the localization of $\tchi^\prime(\psi)$:
		$$
		\norm{e^{-\ve \frac{D_t^2}{2  \tau}}e^{\tau \psi} B_4 u}{L^2}{}\leq \norm{e^{\tau \psi}B_4 u}{L^2}{}\leq \frac{C}{\delta^2} e^{2 \zeta \tau}e^{-c \delta^4 \mu}\norm{u}{H^1}{}.
		$$

		\noindent \textbf{Estimating $B_+$.} Here the derivative does not localize \textit{exactly}. We have:
		\begin{align*}
			\norm{e^{-\ve \frac{D_t^2}{2  \tau}}e^{\tau \psi}B_+ u}{L^2}{}&\leq \norm{e^{-\ve \frac{D_t^2}{2  \tau}}M^\mu_\lambda e^{\tau \psi}B_+ u}{L^2}{}+\norm{e^{-\ve \frac{D_t^2}{2  \tau}}(1-M^\mu_\lambda) e^{\tau \psi}B_+ u}{L^2}{}\\
			&\leq \norm{M^\mu_\lambda e^{\tau \psi}B_+ u}{L^2}{}+C\lambda^{1/2}\left(e^{-\frac{\ve \mu^2}{8 \tau}} +e^{- c \mu}\right)e^{C \frac{\tau^2}{\mu}}e^{\zeta \tau}\norm{u}{H^1}{},
		\end{align*}
		where we have applied successively Lemma~\ref{lemma 2.14 from ll} and  Lemma~\ref{lemma 2.13 from ll}. We estimate now the first term in the above inequality, with $B_+=b_+\partial^\gamma$:
		$$
		\norm{M^\mu_\lambda e^{\tau \psi}B_+ u}{L^2}{}\leq\norm{M^\mu_\lambda e^{\tau \psi}b_+(1-M_\lambda^{2\mu}) \partial^\gamma u}{L^2}{}+\norm{M^\mu_\lambda e^{\tau \psi}b_+M_\lambda^{2\mu} \partial^\gamma u}{L^2}{}.
		$$
		We apply then Lemma~\ref{lemma 2.16 from ll} which gives
		$$
		\norm{M^\mu_\lambda e^{\tau \psi}b_+(1-M_\lambda^{2\mu}) \partial^\gamma u}{L^2}{}\leq \frac{C}{\delta^N} \tau^N e^{C \frac{\tau^2}{\mu}}e^{2 \zeta \tau - c\delta^4 \mu}\norm{u}{H^1}{}.
		$$
		Using the fact that
		$$
		\norm{M^\mu_\lambda e^{\tau \psi}b_+M_\lambda^{2\mu} \partial^\gamma u}{L^2}{}\leq \norm{ e^{\tau \psi}b_+M_\lambda^{2\mu} \partial^\gamma u}{L^2}{}
		$$
		we have thus obtained
		$$
		\norm{e^{-\ve \frac{D_t^2}{2  \tau}}e^{\tau \psi} B_+ u}{L^2}{}\leq \norm{ e^{\tau \psi}B_+M_\lambda^{2\mu} \partial^\gamma u}{L^2}{}+\frac{C}{\delta^N}\mu^{1/2}\tau^N\left(e^{-\frac{\ve\mu^2}{8 \tau}}+e^{\zeta \tau}e^{-c\delta^4\mu}\right)e^{C \frac{\tau^2}{\mu}}e^{\zeta \tau}\norm{u}{H^1}{}.
		$$
		We need finally to estimate $ \norm{ e^{\tau \psi}b_+M_\lambda^{2\mu} \partial^\gamma u}{L^2}{}$.

		A generic term of $B_+$ has the form 
		$$
		B_+=b_+ \partial^\gamma=f \tilde{\tilde{b}}_\lambda\chi_{ \lambda}^+(\psi)\tchi(\psi)\partial^\gamma,
		$$
		where $\tilde{\tilde{b}}=\partial^\beta(\sigma_1)$, $|\beta|\leq 1$. We decompose $\R^{1+n}$ as 
		\begin{align*}
			\R^{1+n}& =O_1\cup O_2, \quad \textnormal{with} \\
			O_1&=\{\psi \notin [\zeta/4, 2 \zeta ] \},\\
			O_2 &= \{\psi\in [\zeta/4,2\zeta]\}.
		\end{align*}
		
		For the region $O_1$ we use the fact that $\chi^+$ is supported in $[\zeta/2, \zeta]$ and then Lemma~\ref{lemma 2.3 from ll} with $f_2=\mathds{1}_{[\zeta/4, 2\zeta]^c}$.
		
		For the region $O_2$ we notice that one can find a smooth $\tilde{\theta}$ with $\tilde{\theta}=1$ on a neighborhood of $O_2$ and supported in $\{\psi>0\}$. We estimate then
		$$
		\norm{e^{\tau \psi}b_+ M^{2\mu}_\lambda \partial^\gamma u}{L^2(O_2)}{}\leq \frac{C}{\delta^2} e^{\zeta \tau } \norm{M^{2\mu}_\lambda \partial^\gamma u}{L^2(O_2)}{}\leq \frac{C}{\delta^2} e^{\zeta \tau } \norm{\tilde{\theta}_\lambda M^{2\mu}_\lambda \partial^\gamma u}{L^2}{}.
		$$
		The final step is to commute $\theta_\lambda$ with $M^{2\mu}_\lambda$. Let $\tilde{\tilde{\theta}} \in C^\infty_0$ be such that $\tilde{\tilde{\theta}}=1$ on a neighborhood of $\supp{\tilde{\theta}}$ and supported in $\{\psi>0\}$. Consequently one has that $\theta=1$ in a neighborhood of $\supp{\tilde{\tilde{\theta}}}$. We use then
		Lemma~\ref{lemma 2.6 from ll} which gives
		$$
		\norm{\tilde{\theta}_\lambda M_\lambda^{2\mu}\partial^\gamma u }{L^2}{}\leq \frac{C}{\delta^4} \norm{\tilde{\tilde{\theta}}_\lambda M_\lambda^{2\mu}u}{H^1}{}+\frac{C}{\delta^4}e^{-c \delta^4 \lambda}\norm{u}{H^1}{},
		$$
		and then Lemma~\ref{lemma 2.11 from ll}
		$$
		\norm{\tilde{\tilde{\theta}}\lambda M_\lambda^{2\mu}u}{H^1}{}\leq  \norm{M^{2\mu}_\lambda \theta_\lambda u}{H^1}{}+\frac{C}{\delta^2}  e^{-c \mu}\norm{u}{H^1}{}.
		$$
		This concludes the estimate for $B_+$.
		
		\noindent \textbf{Estimating $B_3$.} This is the same as for $B_+$ and we omit it.
		
		\bigskip
		
		Combining the estimates for the terms $B_{*}$ concludes the proof of Lemma~\ref{lem_commut}.
	\end{proof}
	
	We can now combine the commutator estimate of Lemma~\ref{lem_commut} with the explicit Carleman estimate of Theorem~\ref{thm_explicit_carleman} to obtain the following.
	
	\begin{lem}
		\label{lem_application_carl}
		With the above notations, there exist $C,c,N>0$ depending on $r_0,n$ only such that for  any $\theta \in C^\infty_0(\R^{1+n})$ with $\theta(x)=1 $ on a neighborhood of $\{\phi \geq \g \} $ one has
		\begin{align}
			&\tau^{1/2} \norm{\pp e^{\tau \psi}\sigma_2 \sigma_{1, \lambda}\tchi(\psi)\chil (\psi)u }{H^1_\tau}{}\\
			&\leq  \frac{C}{\delta^N}  \mu^{1/2}e^{\tau^2/\mu}e^{\zeta \tau } \norm{\Box u}{L^2(\Omega_\delta)}{}+\frac{C}{\delta^{N}} e^{2 \zeta \tau} \norm{M^{2 \mu}_\lambda \theta_\lambda u}{H^1}{}\\& \quad +\frac{C}{\delta^N}\mu^{1/2}\tau^N \left(e^{-8 \zeta \tau}+e^{-\frac{\ve \mu^2}{8 \tau}}+e^{-c\delta^4 \mu}e^{\zeta \tau} \right) e^{C \tau^2/ \mu}e^{\zeta \tau} \norm{u}{H^1}{},  
		\end{align}
		for all $u \in C^\infty_0(\R^{1+n})$ with $\supp(u)\subset \{r\geq \rtt_0\}$, $\lambda \sim \mu$ and $\tau \geq \frac{C}{\delta^8}$.
	\end{lem}
	
	\begin{proof}
		We apply the Carleman estimate of Theorem~\ref{thm_explicit_carleman} to $\sigma_2 \sigma_{1, \lambda}\tchi(\psi)\chil (\psi)u$ to find
		\begin{align}
			&  \tau^{1/2} \norm{\pp e^{\tau \psi}\sigma_2 \sigma_{1, \lambda}\tchi(\psi)\chil (\psi)u }{H^1_\tau}{} \\ &\leq \frac{C}{\delta^{1/2}} \left( \norm{\pp e^{\tau \psi} \Box \sigma_2 \sigma_{1, \lambda}\tchi(\psi)\chil (\psi)u }{L^2}{}+e^{-8\zeta \tau} \norm{e^{\tau \psi}\sigma_2 \sigma_{1, \lambda}\tchi(\psi)\chil (\psi)u}{H^1}{}\right).
		\end{align}
		Using~\eqref{estim_for_carl} and Lemma~\ref{lem_commut} yields
		\begin{align}
			& \norm{\pp e^{\tau \psi} \Box \sigma_2 \sigma_{1, \lambda}\tchi(\psi)\chil (\psi)u }{L^2}{} \leq C \mu^{1/2}e^{\tau^2/\mu}e^{\zeta \tau } \norm{\Box u}{L^2(\Omega_\delta)}{}\\
			&+\frac{C}{\delta^4} e^{2 \zeta \tau} \norm{M^{2 \mu}_\lambda \theta_\lambda u}{H^1}{}+\frac{C}{\delta^N}\mu^{1/2}\tau^N \left(e^{-8 \zeta \tau}+e^{-\frac{\ve \mu^2}{8 \tau}}+e^{-c\delta^4 \mu}e^{\zeta \tau} \right) e^{C \tau^2/ \mu}e^{\zeta \tau} \norm{u}{H^1}{}.
		\end{align}
		Lemma~\ref{lemma 2.13 from ll} gives for the other term
		\begin{align}
			e^{-8\zeta \tau} \norm{e^{\tau \psi}\sigma_2 \sigma_{1, \lambda}\tchi(\psi)\chil (\psi)u}{H^1}{} \leq \frac{C}{\delta^2} \mu^{1/2}\tau e^{-7\zeta \tau}e^{C\frac{\tau^2}{\mu}}\norm{u}{H^1}{}.
		\end{align}
		This proves Lemma~\ref{lem_application_carl} up to taking the worst power of $\delta$ and renaming $N$.
	\end{proof}
	
	\noindent \textbf{The complex analysis argument}
	
	\bigskip

	This corresponds to Section 3.3 in~\cite{Laurent_2018}. The final step of the proof of Theorem~\ref{thm local quant estimate} consists in transferring the estimate provided by Lemma~\ref{lem_commut} from $$\norm{\pp e^{\tau \psi}\sigma_{2}\sigma_{1, \lambda}\tchi(\psi)\chi_{\lambda}(\psi)u}{H^1_\tau}{}$$ to $\norm{M^{\beta\mu}_{c_1 \mu}\sigma_{c_1 \mu}u}{H^1}{}$.
	
	\begin{lem}
		\label{lem_complex_anal}
		Under the above assumptions, there exist $C, N >0$ depending on $r_0,n$ only such that for any $\kappa, c_1>0$, we have that for any $0<\beta \leq \min (\delta^{N}, \kappa \delta^{N}, c_1 \delta^{N}) $, for all $\mu \geq \frac{C}{ \delta^8 \beta}$ and $u \in C^\infty_0(\R^{1+n})$ with $\supp (u) \subset \{r\geq \rtt_0\}$, one has:
		\begin{align*}
			\norm{M^{\beta \mu}\sigma_{2}\sigma_{1, \lambda}\tchi(\psi)\chi_{\lambda}(\psi)\eta_{\lambda}(\psi)u}{H^1}{}
			& \leq \frac{C}{\delta^{N}}e^{-c(\delta, \kappa,c_1) \mu}\left(e^{\kappa \mu}\left(\norm{M^{2\mu}_\lambda\theta_\lambda u}{H^1}{} +\norm{\Box u }{}{}\right)+\norm{u}{H^1}{} \right)
		\end{align*}
		with $\lambda=2 c_1 \mu,$ and 
		$$
		\eta \in C^{\infty}_0((-4\zeta,\zeta)), \quad \eta=1 \: \textnormal{in } \: [-1/2 \zeta,1/2 \zeta].
		$$
		The constant $c(\delta, \kappa,c_1)$ satisfies $c(\delta, \kappa, c_1)\sim \min(\delta^{N}, \kappa \delta^{N}, c_1 \delta^{N})$.
	\end{lem}
	
	\begin{proof}
	For any test function $f \in \mathscr{S}(\R^{1+n})$ we define the distribution
		\begin{align}
			\langle h_f, w \rangle_{\mathcal{E}^\prime(\R), C^{\infty}(\R)}&= \langle M^{\beta \mu}f \sigma_{2}\sigma_{1, \lambda}\tchi(\psi)\chi_{\lambda}(\psi)u, w(\psi)\rangle_{\mathcal{E}^\prime(\R^{1+n}),C^\infty(\R^{1+n})} \\
			&= \langle M^{\beta \mu} \sigma_{2}\sigma_{1, \lambda}\tchi(\psi)\chi_{\lambda}(\psi)w(\psi)u, f\rangle_{\mathscr{S}^\prime(\R^{1+n}),\mathscr{S}(\R^{1+n})}
		\end{align}

		We work with $w=\eta_{ \lambda}$ and estimate the quantity $
		\langle h_f, \eta_{ \lambda} \rangle_{\mathcal{E}^\prime(\R), C^{\infty}(\R)}.
		$
		We have
		\begin{align}
			\hat{h}_f(\vs)=\langle e^{-i \vs \psi} \sigma_{2}\sigma_{1, \lambda}\tchi(\psi)\chi_{\lambda}(\psi)u,  M^{\beta \mu}f \rangle_{\mathcal{E}^\prime(\R), C^{\infty}(\R)}, \quad \vs \in \mathbb{C},
		\end{align}
		which gives the following a priori estimate
		\begin{align}
			\label{apriori}
			|\hat{h}_f(\vs)|\leq C \langle |\vs|\rangle e^{|\Im \vs| \norm{\psi}{L^\infty}{} } \norm{u}{H^1}{} \norm{f}{H^{-1}}{}, \quad \vs \in \mathbb{C}.
		\end{align}
		Using Lemma~\ref{lem_application_carl} we get for all $\tau \geq \frac{C}{\delta^8}$, $\mu \geq 1$ and $\lambda\sim \mu$
		\begin{align}
			\label{estim_imaginary}
			|\hat{h}_f(i \tau)|&= |\langle e^{\frac{\ve D^2_t}{2 \tau}} M^{\beta \mu}f, \pp e^{-\tau \psi} \sigma_{2}\sigma_{1, \lambda}\tchi(\psi)\chi_{\lambda}(\psi)u \rangle |  \nonumber \\
             &\leq \frac{C}{\delta^N} e^{\frac{\ve}{2 \tau} \beta^2 \mu^2} \norm{f}{H^{-1}}{} \bigg(  \mu^{1/2}e^{\tau^2/\mu}e^{\zeta \tau } \norm{\Box u}{L^2(\Omega_\delta)}{}+ e^{2 \zeta \tau} \norm{M^{2 \mu}_\lambda \theta_\lambda u}{H^1}{}\\& \quad +\mu^{1/2}\tau^N (e^{-8 \zeta \tau}+e^{-\frac{\ve \mu^2}{8 \tau}}+e^{-c\delta^4 \mu}e^{\zeta \tau} ) e^{C \tau^2/ \mu}e^{\zeta \tau} \norm{u}{H^1}{}\bigg).
		\end{align}
		Choosing $\lambda=2 c_1 \mu$ and defining, for $\kappa>0$, 
		\begin{align}
			D=e^{\kappa \mu} \left( \norm{M^{2 \mu}_\lambda \theta_\lambda u}{H^1}{}+\norm{\Box u }{L^2}{} \right),
		\end{align}
		we find
		\begin{align}
			|\hat{h}_f(i \tau)|\leq\frac{C}{\delta^N} \mu^{1/2}\tau^N e^{\frac{\ve}{2 \tau} \beta^2 \mu^2}  e^{C \tau^2/ \mu}e^{2\zeta \tau} (D+\norm{u}{H^1}{}) \norm{f}{H^{-1}}{} \left(e^{-\kappa \delta^4 \mu}+ e^{-\frac{\ve \mu^2}{8 \tau}}+ e^{-9\zeta \tau} \right).
		\end{align}

		We look at the quantity of interest and write thanks to Plancherel
		\begin{align}
			\langle h_f, \eta_{ \lambda} \rangle_{\mathcal{E}^\prime(\R), C^{\infty}(\R)}=\frac{1}{2 \pi} \int_\R \hat{h}_f(\vs) \hat{\eta}_\lambda(-\vs) d \vs.
		\end{align}
		Recalling that $\eta \in C^\infty_0(-4\zeta,\zeta)$ and thanks to the Paley-Wiener theorem we have for $\Im \vs \geq 0$
		\begin{align}
			&|\hat{\eta} (-\vs)| \leq C e^{4 \zeta \Im \vs}, \\
			\label{apriori two}
			& |\hat{\eta}_\lambda (-\vs)|= |e^{-\frac{\vs^2}{\lambda}} \hat{\eta}(-\vs)| \leq C e^{\frac{(\Im \vs)^2- (\Re  \vs)^2}{\lambda}}e^{4 \zeta \Im \vs}.
		\end{align}
		We now split, for $0< d \leq 1$ to be chosen later the integral
		\begin{align}
			\int_\R \hat{h}_f(\vs) \hat{\eta}_\lambda(-\vs) d \vs=I_-+I_0+I_+,
		\end{align}
		where $I_-=\int_{-\infty}^{- d \mu} \hat{\eta}_\lambda(-\vs) d \vs $, $I_0=\int_{-d \mu}^{d \mu} \hat{\eta}_\lambda(-\vs) d \vs $, $I_+=\int_{d \mu}^{+ \infty} \hat{\eta}_\lambda(-\vs) d \vs$.
		
		Using~\eqref{apriori} and~\eqref{apriori two} we find, for $\mu \geq 1$ and $\lambda=2 c_1 \mu$
		\begin{align}
			\label{simple_estim}
			|I_\pm|\leq C  \mu e^{-c d^2 \mu} \norm{u}{H^1}{} \norm{f}{H^{-1}}{}.
		\end{align}
		We now focus on $I_0$ and define
		\begin{align}
			\label{def_c0}
			c_0=\frac{C}{\delta^N}(D+\norm{u}{H^1}{})\norm{f}{H^{-1}}{},
		\end{align}
		where the constant $C$ is the one in~\eqref{estim_imaginary}. We further define
		\begin{align}
			H=\frac{1}{c_0} \mu^{-1/2}(\vs+i)^{-N}e^{i2\zeta \vs} \hat{h}_f(\vs),
		\end{align}
		so that~\eqref{estim_imaginary} gives, for $\tau \geq \frac{C}{\delta^8}$, $\mu\geq 1$ and $\lambda=2 c_1 \mu$
		\begin{align}
			H(i \tau) \leq e^{\frac{\ve}{2 \tau} \beta^2 \mu^2}  e^{C \tau^2/ \mu} \left(e^{-\kappa \delta^4 \mu}+ e^{-\frac{\ve \mu^2}{8 \tau}}+ e^{-9\zeta \tau} \right).
		\end{align}
		Thanks to~\eqref{apriori} the function $H$ has sub exponential growth so we can apply Lemma~\ref{lem_harmonic_append} which gives the existence of $M>0$ depending only on $r_0,n$ such that for any $0<d\leq \min (\delta^M, \kappa \delta^M)$, any $0<\beta \leq d \delta^M $ and for all $\mu \geq \frac{C}{\delta^8 \beta}$ we have
		\begin{align}
			|H(\vs)| \leq e^{-8 \zeta \Im \vs}, \quad \textnormal{on } \overline{Q_1} \cap \{d\mu/4 \leq |\vs| \leq 2 d \mu\},
		\end{align}
		where $Q_1=\R^*_+ +i\R^*_+$. Due to symmetry this estimate holds as well in $\R^*_+ +i\R^*_+$. Recalling then the definition of $H$ we obtain the estimate
		\begin{align}
			|\hat{h}_f(\vs)|\leq c_0 \mu^{1/2} \langle |\vs|\rangle^N e^{-6 \zeta \Im \vs}\leq C c_0 d^N \mu^{N+1/2}e^{-6 \zeta \Im \vs},
		\end{align}
		on $\mathbb{C_+}\cap \{d\mu/4 \leq |\vs| \leq 2 d \mu\} $.
		
		We now have explicit dependence of all the constants with respect to the parameter $\delta>0$ and we can conclude the proof. Using the notation of~\cite[Lemma 3.10]{Laurent_2018} we change the contour from $[-d \mu, d \mu]$ to $\Gamma_\pm^V, \Gamma^H$ and obtain the estimates
		\begin{align}
			& \hat{h}_f(\vs) \hat{\eta}_\lambda(-\vs) \leq C c_0d^N \mu^{N+1/2}e^{-2 \zeta \Im \vs}e^{-\frac{3 d^2 \mu}{8c_1}}, \quad \vs \in \Gamma_+^V \cup \Gamma_-^V,\\
			& \hat{h}_f(\vs) \hat{\eta}_\lambda(-\vs) \leq C c_0d^N \mu^{N+1/2}e^{-\zeta d \mu}e^{\frac{d^2}{8c_1}\mu}, \quad \vs \in\Gamma^H.
		\end{align}
		We now fix $d=\min(4c_1 \zeta,d_0)$ which implies
		\begin{align}
			e^{-\zeta d \mu}e^{\frac{d^2}{8c_1}\mu}\leq e^{-cd\zeta \mu}.
		\end{align}
		Combining the estimates in the three regions $\Gamma_\pm^V, \Gamma^H$ together with~\eqref{simple_estim} and recalling the definition of $c_0$ in~\eqref{def_c0} we find that for any $0<\beta\leq d \delta^M$ and all $\mu \geq \frac{C}{\delta^8 \beta}$
		\begin{align}
			&|\langle(M^{\beta \mu} \sigma_{2}\sigma_{1, \lambda}\tchi(\psi)\chi_{\lambda}(\psi)\eta_\lambda(\psi) u, f\rangle_{\mathscr{S}^\prime(\R^{1+n}),\mathscr{S}(\R^{1+n})}| \\
			&\leq \frac{1}{2 \pi} \left(|I_+|+|I_-|+|I_0|\right) \leq \frac{C}{\delta^N}  d^N \mu^{N+1/2} e^{-cd\zeta \mu}(D+\norm{u}{H^{1}}{}) \norm{f}{H^{-1}}{}.
		\end{align}
		We can then choose $d=\min (\delta^M, \kappa \delta^M, c_1 \delta^M)$ and $\beta\leq \min (\delta^{2M}, \kappa \delta^{2M}, c_1 \delta^{2M}) $ for $M$ sufficiently large depending on $r_0,n$ only. We have that for $\mu \geq \frac{C}{\delta^8 \beta}$
		\begin{align}
			|\langle M^{\beta \mu} \sigma_{2}\sigma_{1, \lambda}\tchi(\psi)\chi_{\lambda}(\psi)\eta_\lambda(\psi) u, f\rangle_{\mathscr{S}^\prime(\R^{1+n}),\mathscr{S}(\R^{1+n})}| \leq \frac{C}{\delta^{N_1}}e^{-c(\delta, \kappa,c_1) \mu}(D+\norm{u}{H^{1}}{}) \norm{f}{H^{-1}}{},
		\end{align}
		where $c(\delta, \kappa, c_1)\sim \min(\delta^{N_2}, \kappa \delta^{N_2}, c_1 \delta^{N_2})$ for some $C,N_1,N_2$ depending on $r_0,n$ only. This concludes by duality the proof of Lemma~\ref{lem_complex_anal}, up to taking the worst power of $\delta$ and renaming the constants.
	\end{proof}
	
\begin{proof}[End of the proof of Theorem~\ref{thm local quant estimate}.] In the previous lemmas we have obtained all the key estimates with explicit constants with respect to $\delta$.  We can then conclude exactly as in \cite{Laurent_2018}. Let us sketch the end of the proof for the sake of completeness.

We have:
\begin{align*}
	\norm{M^{\frac{\beta \mu}{2}}_\lambda \sigma_{2}\sigma_{1, \lambda}\tchi(\psi)\chi_{\lambda}(\psi)\eta_{\lambda}(\psi)u}{H^1}{} &\leq \norm{M^{\frac{\beta \mu}{2}}_\lambda (1-M^{\beta \mu})\sigma_{2}\sigma_{1, \lambda}\tchi(\psi)\chi_{\lambda}(\psi)\eta_{\lambda}(\psi)u}{H^1}{}\\
	&\hspace{4mm}+\norm{M^{\frac{\beta \mu}{2}}_\lambda M^{\beta \mu}\sigma_{2}\sigma_{1, \lambda}\tchi(\psi)\chi_{\lambda}(\psi)\eta_{\lambda}(\psi)u}{H^1}{}.
\end{align*}
To control the first term we use Lemma~\ref{lemma 2.3 from ll}. For the second one we use Lemma~\ref{lem_complex_anal}. We find that for any $\kappa, c_1>0$, we have that for any $0<\beta \leq  \min (\delta^{N}, \kappa \delta^{N}, c_1 \delta^{N})$, for all $\mu \geq \frac{C}{ \delta^8 \beta}$ one has:
		\begin{align*}
			\norm{M^{\frac{\beta \mu}{2}}_\lambda \sigma_{2}\sigma_{1, \lambda}\tchi(\psi)\chi_{\lambda}(\psi)\eta_{\lambda}(\psi)u}{H^1}{}
			& \leq \frac{C}{\delta^{N}}e^{-c(\delta, \kappa,c_1) \mu}\left(e^{\kappa \mu}\left(\norm{M^{2\mu}_\lambda\theta_\lambda u}{H^1}{} +\norm{\Box u }{}{}\right)+\norm{u}{H^1}{} \right)
		\end{align*}

Next we combine the above estimate with Lemma~\ref{lemma 2.11 from ll}:
\begin{align}
	\label{aux quant from ll}
	\norm{M^{\frac{\beta \lambda}{4}}\sigma_{ \lambda}u}{H^1}{}&\leq \frac{C}{\delta^2} \bigg( \norm{\sigma_{ \lambda} M^{\frac{\beta \lambda}{2}}u}{H^1}{}+e^{-c \mu }\norm{u}{H^1}{} \bigg) \nonumber \\
	&\leq  \frac{C}{\delta^2} \bigg( \norm{\sigma_\lambda M^{\frac{\beta \mu}{2}}_\lambda \sigma_{2}\sigma_{1, \lambda}\tchi(\psi)\chi_{\lambda}(\psi)\eta_{\lambda}(\psi)u}{H^1}{} \nonumber \\
	&\hspace{4mm}+\norm{\sigma_{\lambda}M^{\frac{\beta \mu}{2}}_\lambda \left(1-  \sigma_{2}\sigma_{1, \lambda}\tchi(\psi)\chi_{\lambda}(\psi)\eta_{\lambda}(\psi)\right)u}{H^1}{}+e^{-c \mu }\norm{u}{H^1}{} \bigg)\nonumber \\
	&\leq \frac{C}{\delta^{N}}e^{-c(\delta, \kappa,c_1) \mu}\left(e^{\kappa \mu}\left(\norm{M^{2\mu}_\lambda\theta_\lambda u}{H^1}{} +\norm{\Box u }{}{}\right)+\norm{u}{H^1}{} \right) \nonumber \\
	&\hspace{4mm}+ \frac{C}{\delta^2} \norm{\sigma_{\lambda}M^{\frac{\beta \mu}{2}}_\lambda \left(1-  \sigma_{2}\sigma_{1, \lambda}\tchi(\psi)\chi_{\lambda}(\psi)\eta_{\lambda}(\psi)\right)u}{H^1}{},
\end{align}
We need now to control the last term in~\eqref{aux quant from ll}. Recall that $\sigma $ has been chosen such that $\sigma \in  C^\infty_0(\{\g-\frac{\zeta}{4}\leq \phi  \leq \g+\frac{\zeta}{4}\})$ and $\sigma=1$ in a neighborhood of $\{\g-\frac{\zeta}{8}\leq \phi  \leq \g+\frac{\zeta}{8}\}$. As a consequence we have that $\sigma_1=\chi(\psi)=\tchi(\psi)=\eta(\psi)=1$ on a neighborhood of $\supp{\sigma}$. We take $\Pi \in C^{\infty}_0$ with $\Pi=1$ on a neighborhood of $\supp{\sigma}$ and such that  $\sigma_{2}=\sigma_1=\chi(\psi)=\tchi(\psi)=\eta(\psi)=1$ on a neighborhood of $\supp{\Pi}$. Then
\begin{multline}
	\label{quant aux inter}
	\norm{\sigma_{\lambda}M^{\frac{\beta \mu}{2}}_\lambda \left(1-  \sigma_{2}\sigma_{1, \lambda}\tchi(\psi)\chi_{\lambda}(\psi)\eta_{\lambda}(\psi)\right)u}{H^1}{}  \\
	\leq \norm{\sigma_{\lambda}M^{\frac{\beta \mu}{2}}_\lambda \left(1-  \sigma_{2}\sigma_{1, \lambda}\tchi(\psi)\chi_{\lambda}(\psi)\eta_{\lambda}(\psi)\right)(1-\Pi)u}{H^1}{}  \\
	+\norm{\sigma_{\lambda}M^{\frac{\beta \mu}{2}}_\lambda \left(1-  \sigma_{2}\sigma_{1, \lambda}\tchi(\psi)\chi_{\lambda}(\psi)\eta_{\lambda}(\psi)\right)\Pi u}{H^1}{}.
\end{multline}
Using the support properties of $\sigma, \sigma_2,\sigma_1, \tchi, \chi, \eta$ and $\Pi$, $1-\Pi$ we find
\begin{align}
	\label{quant aux 1}
	\norm{\sigma_{\lambda}M^{\frac{\beta \mu}{2}}_\lambda \left(1-  \sigma_{2}\sigma_{1, \lambda}\tchi(\psi)\chi_{\lambda}(\psi)\eta_{\lambda}(\psi)\right)(1-\Pi)u}{H^1}{} \leq \frac{C}{\delta^4}e^{-c \delta^4 \mu} \norm{u}{H^1}{},
\end{align}
and
\begin{align}
	\label{quant aux 2}
	&\norm{\sigma_{\lambda}M^{\frac{\beta \mu}{2}}_\lambda \left(1-  \sigma_{2}\sigma_{1, \lambda}\tchi(\psi)\chi_{\lambda}(\psi)\eta_{\lambda}(\psi)\right)\Pi u}{H^1}{}\\
    &\leq \norm{M^{\frac{\beta \mu}{2}}_\lambda \left(1-  \sigma_{2}\sigma_{1, \lambda}\tchi(\psi)\chi_{\lambda}(\psi)\eta_{\lambda}(\psi)\right) \Pi u}{H^1}{} \leq \frac{C}{\delta^4}e^{-c \delta^4 \mu} \norm{u}{H^1}{}.
\end{align}

Combining~\eqref{aux quant from ll}, \eqref{quant aux inter}, \eqref{quant aux 1}, \eqref{quant aux 2} and recalling that $\lambda=2c_1\mu$ yields that for any $\kappa, c_1>0$, we have that for any $0<\beta \leq  \min (\delta^{N}, \kappa \delta^{N}, c_1 \delta^{N})$, for all $\mu \geq \frac{C}{ \delta^8 \beta}$ and $c(\delta, \kappa, c_1)\sim \min(\delta^{N}, \kappa \delta^{N}, c_1 \delta^{N})$
\begin{equation*}
			\norm{M^{\beta \mu}_{c_1 \mu} \sigma_{c_1\mu}u}{H^1}{} \leq \frac{C }{\delta^{N}} e^{\kappa \mu}\left(\norm{M^\mu_{c_1 \mu}\theta_{c_1 \mu}u}{H^1}{}+\norm{\Box u}{L^2(\Omega_\delta)}{}\right)+\frac{C}{\delta^{N}}e^{-c(\delta, \kappa,c_1) \mu } \norm{u}{H^1}{}.
		\end{equation*}
Let finally $\kappa, \alpha>0$. Using the above property for $\mu=\alpha \mu^\prime, c_1=1/\alpha $ and $\kappa$ replaced by $\kappa/\alpha$ we obtain that for $\mu^\prime \geq \frac{C}{\delta^8 \beta \alpha}$
\begin{equation*}
			\norm{M^{\beta\alpha \mu^\prime}_{ \mu^\prime} \sigma_{\mu^\prime }u}{H^1}{} \leq \frac{C }{\delta^{N}} e^{\kappa  \mu^\prime}\left(\norm{M^{\alpha \mu^\prime}_{\mu^\prime }\theta_{\mu^\prime }u}{H^1}{}+\norm{\Box u}{L^2(\Omega_\delta)}{}\right)+\frac{C}{\delta^{N}}e^{-c(\delta, \kappa/\alpha,1/\alpha) \alpha \mu^\prime } \norm{u}{H^1}{}.
		\end{equation*}
where $c(\delta, \kappa/\alpha,1/\alpha) \alpha \sim \min(\delta^{N}, \kappa \delta^{N},  \alpha \delta^{N}) $. We then choose $\beta=  \min (\delta^{N}, \kappa \delta^{N}/\alpha, \delta^{N}/\alpha)$ and define $\beta^\prime=\beta \alpha$, $\kp=c(\delta, \kappa/\alpha,1/\alpha) \alpha$ to find that for any $\kappa, \alpha>0$ there is $\kp \sim \min(\delta^{N}, \kappa \delta^{N},  \alpha \delta^{N}) $ and $\beta^\prime=  \min(\delta^{N}, \kappa \delta^{N},  \alpha \delta^{N})$ such that we have for all $\mu^\prime \geq \frac{C}{\delta^8 \beta^\prime} $
\begin{equation*}
			\norm{M^{\beta^\prime \mu^\prime}_{ \mu^\prime} \sigma_{\mu^\prime }u}{H^1}{} \leq \frac{C }{\delta^{N}} e^{\kappa  \mu^\prime}\left(\norm{M^{\alpha \mu^\prime}_{\mu^\prime }\theta_{\mu^\prime }u}{H^1}{}+\norm{\Box u}{L^2(\Omega_\delta)}{}\right)+\frac{C}{\delta^{N}}e^{-\kp \mu^\prime } \norm{u}{H^1}{}.
		\end{equation*}

This is the desired result up to renaming $\mu^\prime$ to $\mu$ and $\beta^\prime$ to $\beta$.
\end{proof}

\section{Iteration and global estimates}
\label{sec_iteration}

The key point about the local quantitative estimate of Theorem~\ref{thm local quant estimate} is that it can be iterated to a global one, yielding the optimal stability estimate. In our specific context, we start from an estimate which is semiglobal in nature. This makes the iteration process easier to write. To facilitate notation and to make proofs easier to follow we shall use some terminology from~\cite{Laurent_2018} used to describe the iteration procedure. The following definition is motivated by the result of Theorem~\ref{thm local quant estimate}.
\begin{definition}
Fix an open subset $\Omega$ of  $\R^{1+n}=\R_t \times \R^n_x$ and two \textit{finite} collections $(V_j)_{j\in J}$ and $(U_i)_{i \in I}$ of bounded open sets in $\R^{1+n}$. We say that $(V_j)_{j\in J}$ is \textit{under the dependence} of $(U_i)_{i \in I}$, denoted

$$
(V_j)_{j\in J}  \trianglelefteq  (U_i)_{i \in I},
$$
if for any $\theta_i \in C^\infty_0(\R^{1+n})$ such that $\theta_i (t,x)=1$ on a neighborhood of $\overline{U_i}$, for any $\tilde{\theta}_j \in C^\infty_0(V_j)$ and for all $\kappa, \alpha >0$, there exist $C, \kappa^\prime, \beta, \mu_0$ such that for all $\mu \geq \mu_0$ and $u \in C^\infty_0(\R^{1+n})$ with $\supp (u) \subset \{r\geq \rtt_0\}$, one has:

$$
\sum_{j\in J} \norm{M_\mu^{\beta \mu} \tilde{\theta}_{j,\mu} u}{H^1}{} \leq Ce^{\kappa \mu}\left(\sum_{i \in I} \norm {M_\mu^{\alpha \mu} \theta_{i,\mu} u}{H^1}{}+\norm{\Box u}{L^2(\Omega)}{} \right)+Ce^{-\kappa^\prime \mu} \norm{u}{H^1}{}.
$$
\end{definition}

\bigskip
The relation $\trianglelefteq$ is not transitive. That is why the following slightly stronger notion is introduced. 

\begin{definition}
\label{def of triangle strict}
Fix an open subset $\Omega$ of  $\R^{1+n}=\R_t \times \R^n_x$ and two \textit{finite} collections $(V_j)_{j\in J}$ and $(U_i)_{i \in I}$ of bounded open sets in $\R^{1+n}$. We say that $(V_j)_{j\in J}$ is \textit{under the strong dependence} of $(U_i)_{i \in I}$, denoted by

$$
(V_j)_{j\in J}  \vartriangleleft  (U_i)_{i \in I},
$$
if there exist $\tilde{U}_i \Subset U_i $ such that $(V_j)_{j\in J}  \trianglelefteq  (\tilde{U}_i)_{i \in I}$.

\end{definition}

\bigskip

To facilitate the presentation we reproduce here some of the basic properties of the relations $\trianglelefteq$ and $\vartriangleleft$.

\begin{prop}[See Propositions 4.3 and 4.5 in \cite{Laurent_2018}]
\label{prop 4.5 in ll} One has:
\begin{enumerate}
       \item If $V_i \Subset U_i $ for any $i \in I$, then $(V_i)_{i\in I}  \vartriangleleft  (U_i)_{i \in I}$.
    
    \item $\cup_{i \in I} U_i  \trianglelefteq  (U_i)_{i \in I}$.
    
    \item If $V_i  \vartriangleleft  U_i$ for any $i \in I$, then $(V_i)_{i\in I}  \vartriangleleft  (U_i)_{i \in I}$. In particular, if $U_i \vartriangleleft U$ for any $i \in I$, then $(U_i)_{i \in I} \vartriangleleft U$.
    
    \item The relation is transitive, that is:
    
    \begin{equation*}
    [V\vartriangleleft  U \: \textnormal{and} \: U  \vartriangleleft  W    ] \Rightarrow V  \vartriangleleft W.
    \end{equation*}
\end{enumerate}
\end{prop}

The transitivity property is of course of crucial importance since it guarantees that information can indeed be propagated. In our context, we also need to know how the constants behave during the iteration procedure. This is described in the following lemma.

\begin{lem}[Transitivity with explicit constants]
\label{lem_trans_explicit}
Suppose that $U \vartriangleleft W$ with explicit constants given by $C_1, \kappa_1^\prime, \beta_1, \mu_{0,1}$. Suppose as well that $V \vartriangleleft U$ with explicit constants given by $C_2, \kappa_2^\prime, \beta_2, \mu_{0,2}$. Then $V \vartriangleleft W$ and for any $\kappa, \alpha >0$ the explicit constants in the relation $V \vartriangleleft W$ are given by $C_3=C_1 C_2$, $\kappa_3^\prime= \min (\kappa_2^\prime, \kappa_1^\prime/2)$, $\beta_3=\beta_2(\tilde{\kappa},\beta_1)$, $\mu_{0,3}=\max(\mu_{0,1}, \mu_{0,2})$ where $\tilde{\kappa}=\frac{\min(\kappa_1^\prime, \kappa)}{2}$,$\kappa_2^\prime=\kappa_2^\prime(\tilde{\kappa},\beta_1)$, $\kappa_1^\prime=\kappa_1^\prime(\kappa/2, \alpha)$ and $\beta_1=\beta_1(\kappa/2,\alpha)$.
\end{lem}

\begin{proof}
    By definition of $\vartriangleleft$ we can find $\tilde{U} \Subset U$ such that $V \trianglelefteq \tilde{U}$ as well as $\tilde{W} \Subset W$ such that $U \vartriangleleft \tilde{W}$. 

    Consider now $\chi \in C^\infty_0 (U)$ such that $\chi=1$ on $\tilde{U}$. Let $ \kappa, \alpha >0$, $\theta \in C^\infty_0(\R^{1+n})$ with $\theta=1$ on $\tilde{W}$ and consider as well $\ttheta \in C^\infty_0(V)$. Using property $U \trianglelefteq \tilde{W}$  for $\kappa \to \kappa/2 $ and $\alpha \to \alpha$ yields the estimate
    \begin{align}
    \label{trans_lem1}
\norm{M_\mu^{\beta_1 \mu} \chi_{\mu} u}{H^1}{} \leq C_1e^{\kappa/2 \mu}\left(\norm {M_\mu^{\alpha \mu} \theta_{\mu} u}{H^1}{}+\norm{\Box u}{L^2(\Omega_\delta)}{} \right)+C_1e^{-\kappa_1^\prime \mu} \norm{u}{H^1}{},
    \end{align}
    for $\mu \geq \mu_{0,1}$ where the constants $C_1=C_1,\kappa_1^\prime, \beta_1, \mu_{0,1}$ are associated to $( \kappa/2, \alpha)$. We next use the relation $V \trianglelefteq \tilde{U}$ applied for $\alpha\to \beta_1$ and $\kappa \to \frac{\min (\kappa_1^\prime, \kappa)}{2}=\tilde{\kappa}$. This gives 
    \begin{align}
        \label{trans_lem2}
\norm{M_\mu^{\beta_2 \mu} \ttheta_{\mu} u}{H^1}{} \leq C_2 e^{\tilde{\kappa} \mu}\left(\norm {M_\mu^{\beta_1 \mu} \chi_{\mu} u}{H^1}{}+\norm{\Box u}{L^2(\Omega_\delta)}{} \right)+C_2e^{-\kappa_2^\prime \mu} \norm{u}{H^1}{},
    \end{align}
for $\mu \geq \mu_{0,2}$ where the constants $C_2=C_2(\tilde{\kappa}, \beta_1),\kappa_2^\prime, \beta_2, \mu_{0,2}$ are associated to $(\tilde{\kappa}, \beta_1)$.

We can can now combine~\eqref{trans_lem1} and~\eqref{trans_lem2} to get, for $\mu_0 \geq \max (\mu_{0,1}, \mu_{0,2})$
\begin{align}
    \norm{M_\mu^{\beta_2 \mu} \ttheta_{\mu} u}{H^1}{} &\leq C_2 e^{\tilde{\kappa} \mu} \bigg(C_1e^{\kappa/2 \mu}\left(\norm {M_\mu^{\alpha \mu} \theta_{\mu} u}{H^1}{}+\norm{\Box u}{L^2(\Omega_\delta)}{} \right)
   +C_1e^{-\kappa_1^\prime \mu} \norm{u}{H^1}{}+\norm{\Box u}{L^2(\Omega_\delta)}{} \bigg)\\&  \hspace{4mm} +C_2e^{-\kappa_2^\prime \mu} \norm{u}{H^1}{} \\
   &=C_1C_2 e^{(\tilde{\kappa}+\kappa/2) \mu } \norm {M_\mu^{\alpha \mu} \theta_{\mu} u}{H^1}{}+C_2e^{\tilde{\kappa }\mu } \left(1+C_1 e^{\kappa \mu/2}\right) \norm{\Box u }{L^2}{} \\
   &\hspace{4mm}+ \left( C_2 e^{- \kappa_2^\prime \mu} +C_1C_2 e^{(\tilde{\kappa}-\kappa_1^\prime) \mu }\right) \norm{u}{H^1}{} \\
   &\leq 2C_1C_2 e^{\kappa \mu} \left( \norm {M_\mu^{\alpha \mu} \theta_{\mu} u}{H^1}{}+\norm{\Box u}{L^2(\Omega_\delta)}{} \right)+C_1C_2(e^{- \kappa_2^\prime \mu}+e^{-\kappa_1^\prime/2 \mu}  ) \norm{u}{H^1}{},
\end{align}
where for the last inequality we used $\tilde{\kappa} +\kappa/2 \leq \kappa$ and $\tilde{\kappa} - \kappa_1^\prime \leq - \kappa_1^\prime/2$.

The lemma follows by putting together the conditions on the different constants.
\end{proof}

We start with the following proposition which essentially rephrases the result of Theorem~\ref{thm local quant estimate} in a way that will allow to iterate it.

\begin{prop}
    \label{prop_iteration}
 There exist $C,N>0$ depending only on $r_0,n$ such that for  any $\theta \in C^\infty_0(\R^{1+n})$ with $\theta(x)=1 $ on a neighborhood of $\{\phi > \g+\frac{\zeta}{16} \} $, for any $\ttheta \in C^\infty_0(\{\phi > \g- \frac{\zeta}{12}\})$ we have that for all $ \kappa, \alpha>0$ there exist $\kp \sim  \min(\delta^{N}, \kappa \delta^{N},  \alpha \delta^{N})$ and $\beta= \min(\delta^{N}, \kappa \delta^{N},  \alpha \delta^{N})$ such that for all $\mu \geq \frac{C}{ \delta^8 \beta}$ and $u \in C^\infty_0(\R^{1+n})$ with $\supp (u) \subset \{r\geq \rtt_0\}$, one has:
	\begin{equation*}
			\norm{M^{\beta \mu}_{ \mu} \ttheta_{\mu }u}{H^1}{} \leq \frac{C }{\delta^{N}} e^{\kappa  \mu}\left(\norm{M^{\alpha \mu}_{\mu }\theta_{\mu }u}{H^1}{}+\norm{\Box u}{L^2(\Omega_\delta)}{}\right)+\frac{C}{\delta^{N}}e^{-\kp \mu } \norm{u}{H^1}{}.
		\end{equation*}
    \end{prop}

\begin{proof}

Let $\theta \in C^\infty_0(\R^{1+n})$ such that $\theta=1$ on a neighborhood of $\{\phi >\g+\frac{\zeta}{16}\}$, let $\ttheta \in C^\infty_0(\{ \phi > \g-\frac{\zeta}{12}\})$ and $\kappa, \alpha >0$. 

Applying Theorem~\ref{thm local quant estimate} together with Lemma~\ref{lemma 2.11 from ll} gives that for any $\ttheta_1 \in C^\infty_0(\{\g-\frac{\zeta}{8} <\phi < \g+\frac{\zeta}{8} \})$ we have that there exist $\kp \sim  \min(\delta^{N}, \kappa \delta^{N},  \alpha \delta^{N})$ and $\beta= \min(\delta^{N}, \kappa \delta^{N},  \alpha \delta^{N})$ such that for all $\mu \geq \frac{C}{ \delta^8 \beta}$ and $u \in C^\infty_0(\R^{1+n})$ with $\supp (u) \subset \{r\geq \rtt_0\}$, one has
\begin{equation*}
			\norm{M^{\beta \mu}_{ \mu} \ttheta_{1,\mu }u}{H^1}{} \leq \frac{C }{\delta^{N}} e^{\kappa  \mu}\left(\norm{M^{\alpha \mu}_{\mu }\theta_{\mu }u}{H^1}{}+\norm{\Box u}{L^2(\Omega_\delta)}{}\right)+\frac{C}{\delta^{N}}e^{-\kp \mu } \norm{u}{H^1}{}.
		\end{equation*}
The above can be seen as $\{\g -\frac{\zeta}{8} <\phi < \g+\frac{\zeta}{8}\} \vartriangleleft \{ \phi > \g\}$ but with explicit constants. Using again Lemma~\ref{lemma 2.11 from ll} together with the inclusion $\{ \phi > \g +\frac{\zeta}{14}\} \subset \{ \phi > \g +\frac{\zeta}{16}\} $ we have that for $\ttheta_2 \in C^\infty_0(\{ \phi > \g +\frac{\zeta}{14}\})$
\begin{equation*}
			\norm{M^{\alpha \mu/2}_{ \mu} \ttheta_{2,\mu }u}{H^1}{} \leq \frac{C }{\delta^{N}} e^{\kappa  \mu}\left(\norm{M^{\alpha \mu}_{\mu }\theta_{\mu }u}{H^1}{}+\norm{\Box u}{L^2(\Omega_\delta)}{}\right)+\frac{C}{\delta^{N}}e^{-\kp \mu } \norm{u}{H^1}{}.
		\end{equation*}
This is essentially the proof of Property 1 of Proposition~\ref{prop 4.5 in ll} in the particular case of the inclusion $\{ \phi > \g +\frac{\zeta}{14}\} \subset \{ \phi > \g +\frac{\zeta}{16}\} $ where we repeated the proof in order to have explicit constants. 
Adding the two estimates above yields
\begin{equation*}
		\norm{M^{\beta \mu}_{ \mu} \ttheta_{1,\mu }u}{H^1}{}	+ \norm{M^{\alpha \mu/2}_{ \mu} \ttheta_{2,\mu }u}{H^1}{} \leq \frac{C }{\delta^{N}} e^{\kappa  \mu}\left(\norm{M^{\alpha \mu}_{\mu }\theta_{\mu }u}{H^1}{}+\norm{\Box u}{L^2(\Omega_\delta)}{}\right)+\frac{C}{\delta^{N}}e^{-\kp \mu } \norm{u}{H^1}{}.
		\end{equation*}
Noticing that
\begin{align}
\norm{M^{\beta \mu}_{ \mu} \ttheta_{2,\mu }u}{H^1}{} &\leq \norm{M^{\alpha \mu/2}_{ \mu} M^{\beta \mu}_{ \mu} \ttheta_{2,\mu }u}{H^1}{}+\norm{(1-M^{\alpha \mu/2}_{ \mu})M^{\beta \mu}_{ \mu} \ttheta_{2,\mu }u}{H^1}{} \\ 
&\leq  \norm{M^{\alpha \mu/2}_{ \mu} \ttheta_{2,\mu }u}{H^1}{}+\frac{C}{\delta^{N}}e^{-\kp \mu } \norm{u}{H^1}{},
\end{align}
where we used Lemma~\ref{lemma 2.3 from ll}, we find then
    \begin{equation}
    \norm{M^{\beta \mu}_{ \mu} \ttheta_{1,\mu }u}{H^1}{}+ \norm{M^{\beta \mu}_{ \mu} \ttheta_{2,\mu }u}{H^1}{} \leq \frac{C }{\delta^{N}} e^{\kappa  \mu}\left(\norm{M^{\alpha \mu}_{\mu }\theta_{\mu }u}{H^1}{}+\norm{\Box u}{L^2(\Omega_\delta)}{}\right)+\frac{C}{\delta^{N}}e^{-\kp \mu } \norm{u}{H^1}{}.
\end{equation}
This is Property 3 of Proposition~\ref{prop 4.5 in ll} saying that $\{\g -\frac{\zeta}{8} <\phi < \g+\frac{\zeta}{8}\} \vartriangleleft \{ \phi > \g\}$ and $\{ \phi > \g +\frac{\zeta}{14}\} \vartriangleleft \{ \phi > \g \} $ imply that 
\begin{align}
\label{trans1}
\bigg (\{\g -\frac{\zeta}{8} <\phi < \g+\frac{\zeta}{8}\},\{ \phi > \g +\frac{\zeta}{14}\} \bigg) \vartriangleleft \{ \phi > \g \},
\end{align}
again with explicit constants.

We now consider the sets $A=\{ \g-\frac{\zeta}{12} < \phi< \g+\frac{\zeta}{12}\}$ and $B=\{ \phi > \g+\frac{\zeta}{13} \}$. Property 2 of Proposition~\ref{prop 4.5 in ll} says that $A \cup B \trianglelefteq (A,B)$ and proceeding in a similar fashion we have an explicit estimate which is of the same form. Finally, noticing that $A \Subset \{\g -\frac{\zeta}{8} <\phi < \g+\frac{\zeta}{8}\} $ and $B \Subset \{ \phi > \g +\frac{\zeta}{14}\} $ we see that by definition of $\vartriangleleft$  property $A \cup B \trianglelefteq (A,B)$ yields 
$$
A \cup B  \vartriangleleft \bigg (\{\g -\frac{\zeta}{8} <\phi < \g+\frac{\zeta}{8}\},\{ \phi > \g +\frac{\zeta}{14}\} \bigg).
$$
Since $A \cup  B = \{\phi > \g - \frac{\zeta}{12} \}$ we have finally obtained the relation of dependence 
\begin{align}
\label{trans2}
    \{\phi > \g - \frac{\zeta}{12} \} \vartriangleleft \bigg (\{\g -\frac{\zeta}{8} <\phi < \g+\frac{\zeta}{8}\},\{ \phi > \g +\frac{\zeta}{14}\} \bigg).
\end{align}
The proposition then follows (for some different $C$ and $N$) from~\eqref{trans1} and~\eqref{trans2} along with the explicit transitivity property of Lemma~\ref{lem_trans_explicit}.
\end{proof}

Let us define $\b$ by the relation 
$$
\b \delta^2=\frac{\zeta}{12},
$$ 
where we recall that $\zeta$ is defined in~\eqref{def_of_zeta}. The statement of Proposition~\ref{prop_iteration} above can then be written as 
    $$
  \{\phi > \g-\b \delta^2\}  \vartriangleleft \{ \phi > \g \}.
    $$
The point on writing Proposition~\ref{prop_iteration} as we did is that we also state the explicit dependence of the constants hidden in the definition of $\vartriangleleft$ with respect to $\delta$. We can then iterate this result using the transitivity of $\vartriangleleft$ in order to to transfer (low frequency) information from the initial observation set, that is $\{ \phi > \g \}$, to the set $\{\phi > \delta \}$. This corresponds to the property 
    $$
    \{\phi > \delta \}  \vartriangleleft \{ \phi > \g \}.
$$
This is the content of the following proposition, where we additionally have the explicit dependence with respect to $\delta$.

\begin{prop}
\label{prop_low_quant_propagation}
    There exist $C,N>0$ depending only on $r_0,n$ such that for  any $\theta \in C^\infty_0(\R^{1+n})$ with $\theta(x)=1 $ on a neighborhood of $\{\phi > \g+\frac{\zeta}{16} \} $, for any $\ttheta \in C^\infty_0(\{\phi > \delta^2\})$ we have that for all $ \kappa, \alpha>0$ there exist $\kp   \sim \min \left(\frac{\kappa}{2^{k_\delta}}\delta^{N k_\delta} ,\alpha \delta^{N k_\delta}\right)$ and $\beta= \min \left(\frac{\kappa}{2^{k_\delta}}\delta^{N k_\delta} ,\alpha \delta^{Nk_\delta}\right)$ such that for all $\mu \geq \frac{C}{ \delta^8 \beta}$ and $u \in C^\infty_0(\R^{1+n})$ with $\supp (u) \subset \{r\geq \rtt_0\}$, one has:
	\begin{equation*}
			\norm{M^{\beta \mu}_{ \mu} \ttheta_{\mu }u}{H^1}{} \leq \frac{C}{\delta^{Nk_\delta}} e^{\kappa  \mu}\left(\norm{M^{\alpha \mu}_{\mu }\theta_{\mu }u}{H^1}{}+\norm{\Box u}{L^2(\Omega_\delta)}{}\right)+\frac{C}{\delta^{Nk_\delta}}e^{-\kp \mu } \norm{u}{H^1}{},
		\end{equation*}
where $k_\delta= \frac{1}{\delta^2}$.
    \end{prop}

\begin{proof}

Let $\kappa, \alpha >0$ We can suppose without loss of generality that $\kappa, \alpha \leq 1$. We apply Proposition~\ref{prop_iteration} to find  $\{\phi > \g-\b \delta^2\}  \vartriangleleft \{ \phi > \g \}$ with explicit constants. Using the notation of Lemma~\ref{lem_trans_explicit} we have the constants $C_1=\frac{C}{\delta^N}, \kappa_1^\prime(\kappa,\alpha) \sim \min( \kappa \delta^N,\alpha \delta^N)$, $\beta_1(\kappa, \alpha)=\min(\kappa \delta^N, \alpha \delta^N)$ and $\mu_{0,1}=\frac{C}{\delta^8 \beta_1}$. A second application of Proposition~\ref{prop_iteration} yields as well $\{\phi > \g-2\b \delta^2\}  \vartriangleleft \{ \phi > \g-\b \delta^2 \}$ with similar constants. We now apply Lemma~\ref{lem_trans_explicit} to find  $\{\phi > \g-2\b \delta^2\}  \vartriangleleft \{ \phi > \g \}$ and we need to trace the constants. We have $C_3=C_1 C_2= \frac{C^2}{\delta^{2N}}$, $\beta_1=\beta(\kappa/2, \alpha)=\min(\kappa/2 \delta^N, \alpha \delta^N)\sim \kappa_1^\prime$, $\tilde{\kappa}=\frac{\min(\kappa_1^\prime, \kappa)}{2} \sim \min(\kappa/2 \delta^N, \alpha \delta^N)$ and $\kappa_2^\prime=\kappa_2^\prime(\tilde{\kappa},\beta_1)\sim \min (\kappa/2 \delta^{2N}, \alpha \delta^{2N})$. As a consequence $\kappa_3^\prime= \min (\kappa_2^\prime, \kappa_1^\prime/2)\sim \min (\kappa/2 \delta^{2N}, \alpha \delta^{2N})$, $\beta_3= \beta_2(\tilde{\kappa},\beta_1)=\min(\tilde{\kappa} \delta^N, \beta_1 \delta^N)=\min (\kappa/2 \delta^{2N}, \alpha \delta^{2N})$ and $\mu_{0,3}= \max(\mu_{0,1},\mu_{0,2})=\frac{C}{\delta^8 \beta_3}$.

This one step allows to see how the constants evolve after one iteration. We can now continue this procedure to find that after $k$ iterations we propagate
$$
\{ \phi > \g-(k+1) \b \delta^2 \} \vartriangleleft \{ \phi > \g \},
$$
with associated constants
\begin{itemize}
    \item $C_k= \frac{C^{k+1}}{\delta^{N(k+1)}}$,

    \item  $\kappa_k^\prime \sim \min (\frac{\kappa}{2^k}\delta^{N(k+1)} ,\alpha \delta^{N(k+1)})$,

    \item  $\beta_k= \min \left(\frac{\kappa}{2^k}\delta^{N(k+1)} ,\alpha \delta^{N(k+1)}\right)$,

    \item $\mu_{0, k}=\frac{C}{\delta^8 \beta_k}$.
\end{itemize}

Now remark that the number of iterations $k=k_\delta$ needed to reach $\{ \phi > \delta\}$ should satisfy
$$
\g-(k_\delta+1)\b \delta^2 \leq \delta,
$$
which gives $k_\delta \sim \frac{1}{\delta^2}$ and we have the estimate
\begin{equation*}
			\norm{M^{\beta \mu}_{ \mu} \ttheta_{\mu }u}{H^1}{} \leq \frac{C^{k_\delta}}{\delta^{Nk_\delta}} e^{\kappa  \mu}\left(\norm{M^{\alpha \mu}_{\mu }\theta_{\mu }u}{H^1}{}+\norm{\Box u}{L^2(\Omega_\delta)}{}\right)+\frac{C^{k_\delta}}{\delta^{Nk_\delta}}e^{-\kp \mu } \norm{u}{H^1}{}.
		\end{equation*}
This proves Proposition~\ref{prop_low_quant_propagation} up to using that
$\frac{C^{k_\delta}}{\delta^{Nk_\delta}} \leq \frac{C}{\delta^{2Nk_\delta}}$ and renaming $N$.
\end{proof}

\section{Applications}
\label{sec_applications}

We are now ready to prove our main results as stated in Section~\ref{sec_main_results}. The following quantitative estimate is essentially a consequence of Proposition~\ref{prop_low_quant_propagation}, but here we also take into account high frequencies.

We recall the notation $\Omega_\delta= \{\phi > \delta \}$.

\begin{prop}
\label{prop_key}
    
There exist $C,N >0$ depending on $r_0,n$ only such that for $\kappa= \delta^N$ and $\mu_0=C^{k_\delta}$ one has for all  $u \in C^\infty_0(\R^{1+n})$ with $\supp (u) \subset \{r\geq \rtt_0\}$ and all $\mu \geq \mu_0$
\begin{align}
 \norm{ u }{L^2 ( \Omega_{3 \delta})}{} \leq  \frac{C}{\delta^{Nk_\delta}} e^{ \kappa  \mu}\left(\norm{u}{H^1_x(\{\phi > \g\})}{}+\norm{\Box u}{L^2(\Omega_\delta)}{}\right)+\frac{C}{\delta^{Nk_\delta} \mu } \norm{u}{H^1}{},
\end{align}
where $k_\delta=\frac{1}{\delta^2}$
\end{prop}

\begin{proof}
    Let $\chi \in C^\infty_0(\{ \phi > \delta\})$ such that $\chi=1$ on $\{ \phi > 2\delta\}$. Consider as well $\theta \in C^\infty_0(\{ \phi > \g \})$ with $\theta=1$ on $\{ \phi > \g+\zeta/16\}$. We apply then Proposition~\ref{prop_low_quant_propagation} which implies that for any $\kappa >0$ and $\alpha=1$ there exist $\kp \sim \frac{\kappa}{2^{k_\delta}}$ and $\beta=\frac{\kappa}{2^{k_\delta}}$ such that for all $\mu \geq \mu_0 =\frac{C}{ \delta^8 \beta}$ one has:
	\begin{equation}
    \label{application_unfolding}
			\norm{M^{\beta \mu}_{ \mu} \chi_{\mu }u}{H^1}{} \leq \frac{C}{\delta^{Nk_\delta}} e^{\kappa  \mu}\left(\norm{M^{\mu}_{\mu }\theta_{\mu }u}{H^1}{}+\norm{\Box u}{L^2(\Omega_\delta)}{}\right)+\frac{C}{\delta^{Nk_\delta}}e^{-\kp \mu } \norm{u}{H^1}{},
		\end{equation}
with $k_\delta= \frac{1}{\delta^2}$.

We consider $\ttheta \in C^\infty_0(\{ \phi > \g\})$ such that $\ttheta=1$ on $\{\phi > \g+\zeta/16 \}$. Using Lemma~\ref{lemma 2.3 from ll} and recalling the definition of the Fourier multiplier $M^\mu_\mu$ we obtain
\begin{align}
   \norm{M^{ \mu}_{\mu }\theta_{\mu }u}{H^1}{} &\leq   \norm{M^{ \mu}_{\mu } \ttheta \theta_{\mu }u}{H^1}{}+ \norm{M^{ \mu}_{\mu }(1- \ttheta)\theta_{\mu }u}{H^1}{} \\
   &\leq \norm{M^{ \mu}_{\mu } \ttheta \theta_{\mu }u}{H^1}{}+\frac{C}{\delta^N}e^{-c\delta^N \mu} \norm{u}{H^1}{}\\
   \label{thm_ineq2}
   &\leq \frac{C \mu}{\delta^N} \norm{u}{H^1_x(\{\phi > \g\})}{}+\frac{C}{\delta^N}e^{-c\delta^N \mu} \norm{u}{H^1}{}.
\end{align}
We go back to~\eqref{application_unfolding} and choose $\kappa=\frac{c \delta^N}{4} <\frac{c \delta^N}{2} $ to get
\begin{align}
\label{thm_ineq3}
    \norm{M^{\beta \mu}_{ \mu} \chi_{\mu }u}{H^1}{} \leq \frac{C}{\delta^{Nk_\delta}} e^{2 \kappa  \mu}\left(\norm{M^{\mu}_{\mu }\theta_{\mu }u}{H^1}{}+\norm{\Box u}{L^2(\Omega_\delta)}{}\right)+\frac{C}{\delta^{Nk_\delta}}e^{-\kp \mu } \norm{u}{H^1}{},
\end{align}
for some different $N$.

Let now $\tchi \in C^\infty_0(\{ \phi > 2 \delta\})$ with $\tchi=1$ on $\{ \phi > 3 \delta\}$. We use again Lemma~\ref{lemma 2.3 from ll} to bound
\begin{align}
  \norm{\tchi u}{L^2}{} &\leq \norm{\tchi \chi_\mu u }{L^2}{}  +\norm{\tchi (1-\chi_\mu) u }{L^2}{} \leq \norm{\chi_\mu u }{L^2}{}+\frac{C}{\delta^N} e^{-c \delta^N \mu } \norm{u}{H^1}{} \\
  &
  \label{thm_ineq_triangle}
  \leq \norm{M_\mu^{\beta \mu}\chi_\mu u }{L^2}{}+\norm{(1-M_\mu^{\beta \mu})\chi_\mu u }{L^2}{}+\frac{C}{\delta^N} e^{-c \delta^N \mu } \norm{u}{H^1}{}.
\end{align}
The only remaining term to bound is the high frequency term $\norm{(1-M_\mu^{\beta \mu})\chi_\mu u }{L^2}{}$. We write
\begin{align*}
\norm{(1-M^{\beta \mu}_\mu) \chi_\mu u}{L^2}{}\leq  \norm{\frac{\left(1-m_\mu\left(\frac{\xi_t}{\beta \mu}\right)\right)}{|\xi_t|+\langle \xi_x \rangle }}{L^\infty}{} \norm{u}{H^1}{}.
\end{align*}
In the region $|\xi_t| \geq \frac{\beta \mu}{2}$ we simply have
$$
\frac{\left(1-m_\mu\left(\frac{\xi_t}{\beta \mu}\right)\right)}{|\xi_t|+\langle \xi_x \rangle} \leq \frac{C}{ \beta \mu}.
$$
In the region $|\xi_t |\leq \frac{\beta \mu}{2}$ we use the support of $m$ and Lemma~\ref{lemma 2.3 from ll} to get
$$
\frac{\left(1-m_\mu\left(\frac{\xi_t}{\beta \mu}\right)\right)}{|\xi_t|+\langle \xi_x \rangle} \leq C e^{- c \beta^2 \mu}.
$$
In particular we have the estimate
$$
\norm{\frac{\left(1-m_\mu\left(\frac{\xi_t}{\beta \mu}\right)\right)}{|\xi_t|+\langle \xi_x \rangle}}{L^\infty}{} \leq \frac{C}{ \beta \mu},
$$
and therefore
\begin{align}
\label{thm_ineq4}
  \norm{(1-M^{\beta \mu}_\mu) \chi_\mu u}{L^2}{}   \leq \frac{C}{ \beta \mu} \norm{u}{H^1}{}.
\end{align}

Collecting~\eqref{application_unfolding}, \eqref{thm_ineq2}, \eqref{thm_ineq3}, \eqref{thm_ineq_triangle} and~\eqref{thm_ineq4}  yields the estimate
\begin{align}
 \norm{\tchi u }{L^2}{} &\leq \frac{C}{\delta^{Nk_\delta}} e^{2 \kappa  \mu}\left(\norm{u}{H^1_x(\{\phi > \g\})}{}+\norm{\Box u}{L^2(\Omega_\delta)}{}\right)+\left(\frac{C}{\delta^{Nk_\delta}}e^{-\kp \mu }+\frac{C}{\beta \mu }+\frac{C e^{-c \delta^N \mu }}{\delta^N} \right) \norm{u}{H^1}{}  \\
 &\leq  \frac{C}{\delta^{Nk_\delta}} e^{2 \kappa  \mu}\left(\norm{u}{H^1_x(\{\phi > \g\})}{}+\norm{\Box u}{L^2(\Omega_\delta)}{}\right)+\frac{C}{\delta^{Nk_\delta} \mu } \norm{u}{H^1}{}.
\end{align}
Recalling that $\tchi=1$ on $\{ \phi > 3 \delta \}$ finishes the proof, after renaming the constants.
\end{proof}

We recall the notation for the diamond
\begin{align}
   \D=\{ \phi>0\} \cap \{t \in [-\frac{\rtt}{2}, \frac{\rtt}{2}]\}, 
\end{align}
and the cylinder
\begin{align}
    \C=\{r\leq \rtt\} \times [-\frac{\rtt}{2}, \frac{\rtt}{2}].
\end{align}
We are now ready to prove Theorem~\ref{thm_log_close_tocone}.

\begin{proof}[Proof of Theorem~\ref{thm_log_close_tocone}]
    Let $\chi \in C^\infty_0(\R^{1+n})$ with $\supp(\chi) \subset \{r \leq \rtt_0\}$ and $\chi=1$ on $\{r\geq \rtt \} \cap \D $. We start by writing 
    $$
    \norm{u}{L^2(\Omega_\delta)}{}\leq  \norm{\chi u}{L^2(\Omega_\delta)}{}+ \norm{(1-\chi)u}{L^2(\Omega_\delta)}{}
    $$
Using the support properties of $\chi$ and $u$, we can apply Proposition~\ref{prop_key} with $\g=\frac{r^2_0}{8}$ to the first term to find that there exist $C,N >0$ depending only on $r_0,n$ such that for $\kappa= \delta^N$ and $\mu_0=C^{k_\delta}$ one has for all $\mu \geq \mu_0$
\begin{align}
\norm{\chi u}{L^2(\Omega_{3\delta)}}{}  &\leq   \frac{C}{\delta^{Nk_\delta}} e^{ \kappa  \mu}\left(\norm{\chi u}{H^1_x(\{\phi > \g\})}{}+\norm{\Box  (\chi u)}{L^2(\Omega_\delta)}{}\right)+\frac{C}{\delta^{Nk_\delta} \mu } \norm{ \chi u}{H^1}{} \\
 &\leq \frac{C}{\delta^{Nk_\delta}} e^{ \kappa  \mu}\left(\norm{ u}{H^1_x(\C)}{}+\norm{[\Box,  \chi] u}{L^2(\Omega_\delta)}{}+\norm{\Box u}{L^2(\Omega_\delta)}{}\right)+\frac{C}{\delta^{Nk_\delta} \mu } \norm{u}{H^1}{} \\
 &\leq  \frac{C}{\delta^{Nk_\delta}} e^{ \kappa  \mu}\left(\norm{ u}{H^1(\C)}{}+\norm{\Box u}{L^2(\Omega_\delta)}{}\right)+\frac{C}{\delta^{Nk_\delta} \mu } \norm{u}{H^1}{}.
\end{align}
Above we used the fact that for $\g=\frac{\rtt^2}{8}$ one has $\{\phi > \g\} \subset \{r \leq \rtt \}$ and that by standard energy estimates/quantitative finite speed of propagation for the wave equation (see for instance~\cite[Chapter 1.2]{lerner2019carleman}) one has $\norm{u}{H^1_x(\{\phi>0\}\cap \{r \leq \rtt\})}{} \leq \norm{ u}{H^1(\C)}{}+\norm{\Box u}{L^2(\{\phi>0\})}{}$. Notice that by approximation the estimate extends to $u \in H^1(\{\phi>0\})$ such that $\Box u \in L^2$ in $\{\phi>0\}$ (see~\cite[Lemma B.16]{lerner2019carleman}).  Now for the second term we have the loose estimate 
\begin{align}
    \norm{(1-\chi) u}{L^2(\Omega_{3 \delta})}{} \leq \norm{u}{L^2(\C)}{},
\end{align}
which combined with the above and recalling the explicit constants of Proposition~\ref{prop_key} yields the existence of $C,N >0$ such that for $\kappa= \delta^N$ and $\mu_0=C^{k_\delta}$ one has for all $\mu \geq \mu_0$
\begin{align}
  \norm{ u}{L^2(\Omega_{3\delta})}{}  \leq  \frac{C}{\delta^{Nk_\delta}} e^{ \kappa  \mu}\left(\norm{ u}{H^1(\C)}{}+\norm{\Box u}{L^2}{}\right)+\frac{C}{\delta^{Nk_\delta} \mu } \norm{u}{H^1}{},
\end{align}
where we recall that $k_\delta=\frac{1}{\delta^2}$.

We can now apply the optimization Lemma~\ref{lem_optimization} with $a=\norm{ u}{L^2(\D_{3\delta})}{}$, $b=\frac{C}{\delta^{Nk_\delta}}(\norm{ u}{H^1(\C)}{}+\norm{\Box u}{L^2}{})$, $c= \frac{C}{\delta^{Nk_\delta}} \norm{u}{H^1}{}$, $C_1= \kappa=\delta^N$, $C_2=1$ and $D_1=2C_1 \mu_0$ to find 
\begin{align}
\norm{ u}{L^2(\Omega_{3\delta})}{} \leq \frac{D_1}{\delta^{Nk_\delta} } \frac{C \norm{u}{H^1}{}}{\log\left(1+\frac{\norm{u}{H^1}{}}{\norm{u}{H^1(\C)}{}+\norm{\Box u}{L^2}{}}\right)},
\end{align}
which implies for some different constants $C,N >0 $
\begin{align}
\label{estim_in_cone}
\norm{ u}{L^2(\Omega_{3\delta})}{} \leq \frac{C}{\delta^{Nk_\delta} } \frac{\norm{u}{H^1}{}}{\log\left(1+\frac{\norm{u}{H^1}{}}{\norm{u}{H^1(\C)}{}+\norm{\Box u}{L^2}{}}\right)}.
\end{align}

Now to get the estimate restricted in the diamond $\D$ let $u \in H^1(\D)$ such that $\Box u = f \in L^2$ in $\D$. Consider then $\tilde{u} \in H^1(\{\phi >0\})$ satisfying $\Box \tilde{u}=f$ and $\tilde{u}_{|t=0}=u_{|t=0}$, $\partial_t \tilde{u}_{|t=0}=\partial_t u_{|t=0}$. Then by finite speed of propagation we have that $\tilde{u}=u$ on $\D$. The desired result follows by applying~\eqref{estim_in_cone} to $\tilde{u}$, replacing $\delta$ by $\delta^2/3$, renaming $N$ and changing notation from $\rtt$ to $R$.
\end{proof}

\begin{remark}
    \label{rem_dependance_q}
    All the obtained estimates remain valid when one replaces $\Box$ by $\Box+q$, with $q$ a time independent potential. Under the additional assumption that $\norm{q}{L^\infty}{}\leq M$ the constants remain uniform with respect to $M$. This a consequence of Remark~\ref{rem_potential}.
\end{remark}

\bigskip

As already discussed, Theorem~\ref{thm_log_close_tocone} allows to quantify uniqueness $\delta$ close to the diamond $\D$ with an explicit dependence with respect to the parameter $\delta$. By estimating the remaining part $\D \backslash \D_\delta$ we can in fact obtain a stability estimate up to \textit{the whole diamond} $\D$.

\begin{proof}[Proof of Theorem~\ref{thm_up_to_cone}]
    We need to get a bound for the term $\norm{u}{L^2(\D \backslash \D_\delta)}{}$. We then notice that the Lebesgue measure of the set $\D \backslash \D_\delta$ satisfies $|\D \backslash \D_\delta| \leq C \delta^2$ for some $C>0$ depending only on $r_0,n$. We can then use Hölder's inequality as well as the Sobolev embedding $H^1 \hookrightarrow L^{\frac{2n}{n-2}}$ to get
    \begin{align}
        \norm{u}{L^2(\D \backslash \D_\delta)}{2} \leq \norm{u^2}{L^{\frac{n}{n-2}}}{}\norm{\mathds{1}_{\D \backslash \D_\delta}}{L^{3/2}}{}\leq C \delta^{\frac{8}{3}} \norm{u}{L^\frac{2n}{n-2}}{2} \leq C^\prime \delta^{\frac{8}{3}} \norm{u}{H^1}{2}.
    \end{align}

Using Theorem~\ref{thm_log_close_tocone} we obtain then
\begin{align}
    \norm{u}{L^2(\D)}{} &\leq \norm{u}{L^2(\D_\delta)}{}+\norm{u}{L^2(\D \backslash \D_\delta)}{} \leq Ce^{1/\delta^5} \frac{\norm{u}{H^1}{}}{\log\left(1+\frac{\norm{u}{H^1}{}}{\norm{u}{H^1(\C)}{}+\norm{f}{L^2}{}}\right)}+C \delta^{\frac{4}{3}}\norm{u}{H^1}{}.
\end{align}
We can now use again the optimization lemma~\ref{lem_optimization} with $\mu=\frac{1}{\delta^5}$ and $\alpha=\frac{4}{15}$ to find that there exists a constant $C>0$ depending only on $r_0,n$ such that
\begin{align}
    \norm{u}{L^2(\D)}{}  \leq \frac{C \norm{u}{H^1}{}}{\left(\log\left( \frac{\norm{u}{H^1}{}}{b}+1\right)\right)^{4/15}},
\end{align}
where $1/b=\frac{1}{\norm{u}{H^1}{}}\log\left(1+\frac{\norm{u}{H^1}{}}{\norm{u}{H^1(\C)}{}+\norm{f}{L^2}{}}\right)$. This proves the theorem.
\end{proof}

	\appendix
	
	\section{Some lemmata}
    \label{sec_append}
	
	The following is the conclusive lemma for the qualitative unique continuation with the Gaussian weight. See~\cite[Proposition~2.1]{Hor:97}.
	\begin{lem}
		\label{lemme d analyse harmonique}
		Let $u \in L^2(\R^n)$ and let $\phi$ be a smooth real valued function. Let $(A_\tau)_{\tau >0}$ be a family of continuous bounded functions in $\R^n$, such that for any compact set $K\subset \R^n$, we have  $\| A_\tau -1\|_{L^\infty(K)} \rightarrow_{\tau \rightarrow \infty} 0$. If
		$$
		\norm{A_\tau (D)e^{\tau\phi} u}{L^2}{}\leq C, \quad \tau \geq \tau_0 , 
		$$
		then $\supp {u} \subset \{\phi \leq 0\}$.
	\end{lem}

	\begin{lem}[See Lemma 3.20 in~\cite{LL:23notes}]
		\label{lem_almostlocal}
		There exist $C,c>0$ such that for $\chi_1, \chi_2 \in C^\infty(\R^{n+1})$ with all derivatives bounded and satisfying $\dist(\supp(\chi_1),\supp(\chi_2))=d>0$ one has for all $u \in \mathscr{S}(\R^{n+1})$ and $\lambda > 0$
		$$
		\norm{\chi_1 e^{-\frac{D^2_t}{\lambda}} (\chi_2u)}{H^k}{} \leq C \norm{\chi_1}{W^{k,\infty}}{} \norm{\chi_2}{W^{k,\infty}}{}e^{-c d^2 \lambda} \norm{u}{H^k}{}.
		$$
	\end{lem}

\bigskip
	
	We now recall some lemmas from \cite{Laurent_2018} that we use in the article. We state them in a more precise way concerning the dependency of the constants. The proofs follow immediately from the ones in~\cite{Laurent_2018}.

	\begin{lem}[See Lemma 2.3 in \cite{Laurent_2018}]
		\label{lemma 2.3 from ll}
		There exist $C,c >0$ such that for any $f_1, f_2 \in L^{\infty}(\R^{n+1})$ with $\textnormal{dist}(\supp{f_1},\supp{f_2})\geq d>0$ and all $\lambda\geq 0$, we have
		$$
		\norm{f_{1,\lambda}f_2}{L^{\infty}}{}\leq C e^{-c d^2 \lambda}\norm{f_1}{L^\infty}{} \norm{f_2}{L^\infty}{}, \quad \norm{f_{1, \lambda}f_{2,\lambda}}{L^\infty}{} \leq C e^{-c d^2 \lambda}\norm{f_1}{L^\infty}{} \norm{f_2}{L^\infty}{}.
		$$

	\end{lem}

	\begin{lem}[See Lemma 2.4 in \cite{Laurent_2018}]
		\label{lemma 2.4 from ll}
		For every $k \in \mathbb{N}$, there exist $C,c >0$ such that for $f_2 \in C^{\infty}(\R^{n+1})$ with all derivatives bounded, $d>0$, all $f_1 \in H^k(\R^{n+1})$ such that $\textnormal{dist}(\supp{f_1},\supp{f_2})\geq d$ and all $\lambda \geq 0$ we have 
		$$
		\norm{f_{1, \lambda}f_2}{H^k}{} \leq C \norm{f_2}{W^{k,\infty}}{} e^{- c d^2\lambda}\norm{f_1}{H^k}{}.
		$$

	\end{lem}

	\begin{lem}[See Lemma 2.6 in \cite{Laurent_2018}]
		\label{lemma 2.6 from ll}
		There exist $C,c>0$ such that for $f_1, f_2 \in C^{\infty}_0(\R^{n+1})$ with $f_1=1$ in a neighborhood of $\supp{f_2}$ we have for all $\lambda>0$, and all $u \in H^1(\R^{n+1})$,  
		$$
		\norm{f_{2,\lambda}\partial^\alpha u}{L^2}{} \leq C \norm{f_1}{W^{1,\infty}}{} \norm{f_2}{W^{1,\infty}}{} \left(\norm{f_{1, \lambda}u}{H^1}{}+e^{-c d^2 \lambda}\norm{u}{H^1}{}\right), \quad \textnormal{for all} \: |\alpha|\leq 1,
		$$
		where $d=\textnormal{dist}(\supp{f_2},\supp{(1-f_1}))$.
		
	\end{lem}

	\begin{lem}[See Lemma 2.10 in \cite{Laurent_2018}]
		\label{lemma 2.10 from ll}
		For any $k \in \mathbb{N}$, there exist $C,c >0$ such that for $f_1$, $f_2$ in $C^\infty$ bounded as well as their derivatives with $\textnormal{dist}(\supp{f_1},\supp{f_2})\geq d>0$ for all $\mu>0$ and $\lambda>0$, we have
		$$
		\norm{f_{1,\lambda}M^\mu_\lambda f_{2,\lambda}}{H^k\rightarrow H^k}{}\leq C \norm{f_1}{W^{k,\infty}}{}\norm{f_2}{W^{k,\infty}}{} \left(e^{-c d^2 \frac{\mu^2}{\lambda}}+e^{-c d^2 \lambda}\right).
		$$
	\end{lem}

	\begin{lem}[See Lemma 2.11 in \cite{Laurent_2018}]
		\label{lemma 2.11 from ll}
		There exist $C, c>0$ such that for all $f \in C^\infty_0(\R^{n+1})$, $\mu>0$, $\lambda>0$ and $u  \in H^1(\R^{n+1})$, one has
		$$
		\norm{ M^\mu_\lambda  f_\lambda u}{H^1}{} \leq \norm{f_\lambda M^{2\mu}_\lambda u }{H^1}{}+C \norm{f}{W^{1,\infty}}{} \left(e^{-c \frac{\mu^2}{\lambda}}+e^{-c \lambda}\right)\norm{u}{H^1}{}.
		$$ 
    Moreover, there exist $C,c >0$  such that for any $f_1 \in C^\infty_0(\R^{n+1})$ with $\dist(\supp f, \supp (1-f_1))=d >0$ , for
all $\mu, \lambda > 0$ and all $u \in H^1(\R^{n+1})$ we have
\begin{equation*}
    \norm{f_\lambda M^\mu_\lambda u}{H^1}{}\leq C  \norm{f}{W^{1,\infty}}{} \norm{f_{1,\lambda}M^\mu_\lambda u}{H^k}{}+C \norm{f}{W^{1,\infty}}{} \norm{f_1}{W^{1,\infty}}{}(e^\frac{c d^2 \mu^2}{\lambda}+e^{-c d^2\lambda}) \norm{u}{H^1}{}.
\end{equation*}
	\end{lem}

	\begin{lem}[Lemma 2.13 in \cite{Laurent_2018}]
		\label{lemma 2.13 from ll}
		There exists $C>0$ such that for all $D  \in \mathbb{R}$, $\tilde{\chi} \in L^\infty(\R) $ such that $\supp{\tilde{\chi}} \subset (-\infty, D],$ for all $\lambda, \tau >0$, we have
		$$
		\norm{e^{\tau \psi } \tilde{\chi}_\lambda(\psi)}{L^\infty}{}\leq C \norm{\tchi}{L^\infty}{}\langle \lambda \rangle^{1/2}e^{D \tau}e^{\frac{\tau^2}{\lambda}},
		$$
		for all $\psi \in C^0(\R^{n+1}; \R).$
		
	\end{lem}

	\begin{lem}[Lemma 2.14 in \cite{Laurent_2018}]
		\label{lemma 2.14 from ll}
		There exist $C, c>0$ such that, for any $\ve, \tau, \lambda, \mu>0$, for any $k \in \mathbb{N}$, we have:
		$$
		\norm{e^{- \frac{\ve D_t^2}{2\tau}}(1-M_\mu^\lambda)}{H^k\rightarrow H^k}{}\leq e^{-\frac{\ve \mu^2}{8 \tau}}+Ce^{-c \lambda}.
		$$
		
	\end{lem}

	\begin{lem}[See Lemma 2.17 in \cite{Laurent_2018}]
		\label{lemma 2.16 from ll}
		Let $\psi$ be a smooth real valued function on $\R_t \times \R^n_x$, which is a quadratic polynomial in $t$, let $R_\sigma>0$, and $\sigma \in C^\infty_0(\{\phi \geq 2 \delta \})$. Let $\chi \in C^\infty_0(\R)$ with $\supp(\chi)\subset (-\infty,\zeta),$ and $\tchi\in C^\infty_0(\R)$ such that $\tchi=1$ on a neighborhood of $(-\infty,3/2 \zeta)$, $\supp{\tchi}\subset (-\infty,2\zeta)$. Let $f$ be bounded, compactly supported and real analytic in the variable $t$ in a neighborhood of $\overline{\{\phi \geq 2 \delta \}}$ and define 
		$$
		g:=e^{\tau \psi}\chi_{ \lambda}(\psi)\tchi(\psi)f \sigma_\lambda,
		$$
		where $\psi=\phi-\g$, $\g>0$ and $\zeta$ is defined in~\eqref{def_of_zeta}. Then one has the following estimate: for all $c>0$ there exist $c_0, C, N>0$ such that for any $\delta>0$, $\tau, \mu \geq 1$ and $\frac{\mu}{c}  \leq \lambda \leq c \mu,$ we have
		$$
		\norm{M_\lambda^{\mu/2}g(1-M_\lambda^\mu)}{L^2\rightarrow L^2}{}\leq \frac{C}{\delta^N}\tau^N e^{\frac{\tau^2}{\lambda}}e^{2 \zeta \tau}e^{-c_0 \delta^4\mu}.
		$$
	\end{lem}
	\begin{proof}
		The only difference with respect to Lemma 2.17 in~\cite{Laurent_2018} is that $\sigma$ is supported in $\{\phi \geq 2 \delta\}$ instead of a ball, which does not change the proof. The explicit dependence with respect to the parameter $\delta$ follows from the more precise statement given in Lemma 2.15 in~\cite{Laurent_2018}.
	\end{proof}
	
	\bigskip
	
	The following is the complex analysis lemma which corresponds to~\cite[3.11]{Laurent_2018}, which we state here in our context with explicit dependence of $d_0, \beta_0$ on $\delta.$
	
	\begin{lem}
		\label{lem_harmonic_append}
		Let $C_1, C_2> 0$ depending on $r_0,n$ only. There exist $C, M>0$ depending on $r_0,n$ only such that the following property holds for $0<\delta<1$. Let $\kappa>0$ and define $d_0=\min(\delta^M, \delta^M \kappa)$. For any $0<d\leq d_0$, define $\beta_0= d \delta^M$. Then, for any $0<\beta \leq \beta_0$ and all $\mu \geq \frac{C}{\delta^8 \beta}$, the following holds: for every holomorphic function $H$ in $Q_1=\R^*_+ +i\R^*_+$, continuous on $\overline{Q_1}$ and satisfying
		\begin{align}
			&|H(i \tau )| \leq e^{\frac{\ve}{2 \tau} \beta^2 \mu^2}  e^{C_1 \tau^2/ \mu} \max \left(e^{- \kappa \delta^4 \mu}, e^{-\frac{\ve \mu^2}{8 \tau}}, e^{-9\zeta \tau} \right), \\
			& |H(\vs)| \leq e^{C_2 \Im \vs}, \quad \textnormal{on } \overline{Q_1},
		\end{align}
		we have
		\begin{align}
			|H(\vs)| \leq e^{-8 \zeta \Im \vs} \quad \textnormal{on } \overline{Q_1} \cap \{d\mu/4 \leq |\vs| \leq 2 d \mu\}.
		\end{align}
	\end{lem}
	\begin{proof}
		The proof is based on~\cite[Lemma B.2]{Laurent_2018}, here we state the explicit dependence with respect to $\delta$. We recall that in our context $\ve$ is fixed as $\ve=c \delta$ with $c$ depending on $r_0,n$ only and $\zeta=\tilde{c} \delta^2$ with $\tilde{c}$ depending on $r_0,n$ only. By inspecting the proof of~\cite[Lemma B.2]{Laurent_2018} we have the following conditions on $\beta$ and $d$:
		
		\begin{enumerate}
			\item We need $\beta \sqrt{\frac{4}{\zeta}}< \min (\frac{\zeta}{4 C_1}, \frac{\kappa \delta^4}{9 \zeta}, \frac{\sqrt{\ve}}{3 \sqrt{\zeta}})$. This is achieved by taking $\beta \leq \min(\delta^M, \delta^M \kappa)$ with $M$ sufficiently large depending only on $r_0,n$. 
			
			\item  We need $ 2d \leq 2d_0 = \min (\nu D, D/2)$ where 
			\begin{align}
				\nu \geq \zeta^A/D, \quad D=\min \left(\frac{\zeta}{4 C_1}, \frac{\kappa \delta^4}{9 \zeta}, \frac{\sqrt{\ve}}{3 \sqrt{\zeta}} \right),
			\end{align}
			for some $A$ depending only on $r_0,n$. We can then take
			\begin{align}
				d_0 \leq \min(\delta^M, \delta^M \kappa),
			\end{align}
			with $M$ sufficiently large depending only on $r_0,n$. 
			
			\item We have the condition  $\beta \leq d \sqrt{\zeta}/16$ on $\beta$. This is satisfied by taking $\beta\leq d \delta^M$.
		\end{enumerate}
		Combining the above gives the stated conditions on $d$ and $\beta$.
	\end{proof}

        \bigskip

The following is the optimization lemma that gives the stability estimates of Theorem~\ref{thm_log_close_tocone} and Theorem~\ref{thm_up_to_cone}. It is a particular case of~\cite[Lemma~A.3]{Laurent_2018}

\begin{lem}
\label{lem_optimization}
    Suppose that the real numbers $a,b,c >0$ satisfy $b \leq C_2 c$, $a \leq c $ and 
    $$
    a \leq e^{C_1 \mu} b+\frac{c}{\mu^\alpha},
    $$
for all $\mu\geq \mu_0$ and some $\alpha>0$. Then one has 
\begin{align}
    a\leq \frac{D_1}{\log(c/b+1)^\alpha}c,
\end{align}
with $D_1=(2 C_1)^\alpha \max (K,\mu_0)$ and $K=\frac{1}{2 C_1}\sup_{x \leq C_2} (1+x)^{1/2}x^{1/2}\log(1/x+1)$. 
\end{lem}

   \bigskip

\noindent \textbf{Acknowledgements} The authors were supported by the European Research Council 
of the European Union, grant 101086697 (LoCal), 
and the Research Council of Finland, grants 347715,
353096 (Centre of Excellence of Inverse Modelling and Imaging) 
and 359182 (Flagship of Advanced Mathematics for Sensing Imaging and Modelling). 
Views and opinions expressed are those of the authors only and do not 
necessarily reflect those of the European Union or the other funding 
organizations.

	\bibliographystyle{alpha}
\small \bibliography{bibl}

\end{document}